\documentclass{amsproc}[12pt]

\usepackage[ansinew]{inputenc}
\usepackage{graphicx}
\usepackage{color}

\usepackage{enumerate,latexsym}
\usepackage{latexsym}
\usepackage{amsmath,amssymb}% CUIDADO: Si uso estos packages, $\widetilde{}$ no funciona a veces
\usepackage{graphicx}
\newfont{\bb}{msbm10 at 11pt}
\newfont{\bbsmall}{msbm8 at 8pt}

\newtheorem{teorema}{Theorem}

\newtheorem{lema}{Lemma}
\newtheorem{corolario}{Corollary}
\newtheorem{definicion}{Definition}
\newtheorem{observacion}{Remark}
\newtheorem{ejemplo}{Example}
\newtheorem{claim}{Claim}

\def\t{\theta}

\def\lflecha{\longrightarrow}

\newcommand{\ep}{\varepsilon}

\def\cA{\mathcal{A}}
\def\cD{\mathcal{D}}

\def\cK{\mathcal{K}}

\def\cL{\mathcal{L}}
\def\cU{\mathcal{U}}
\def\cM{\mathcal{M}}

\def\cN{\mathcal{N}}
\def\cH{\mathcal{H}}
\def\cW{\mathcal{W}}

\def\war{M^2\times_f \R}
\def\Kext{K_{{\rm ext}}}
\let\8=\infty \let\0=\emptyset  
\let\hat=\widehat

\let\tilde=\widetilde

\let\landa=\lambda
\let\alfa=\alpha

\let\parc=\partial
\let\ep=\varepsilon

\def\landa{\lambda}

\def\flecha{\rightarrow}
\def\esiz{\langle}
\def\esde{\rangle}

\def\cte.{\mathop{\rm cte.}\nolimits}

 \def\Im{\mathop{\rm Im }\nolimits}

\def\cosh{\mathop{\rm cosh }\nolimits}

 \def\ext{\mathop{\rm ext }\nolimits}

\def\M{\mathbb{M}}
\def\Z{\mathbb{Z}}
\def\A{\mathbb{A}}

\def\R{\mathbb{R}}

\def\C{\mathbb{C}}
\def\D{\mathbb{D}}

\def\H{\mathbb{H}}
\def\S{\mathbb{S}}

\newcommand{\beq}{\begin{equation}}
\newcommand{\eeq}{\end{equation}}

\numberwithin{equation}{section}

\begin{document}

\begin{title}[Isolated singularities of graphs in warped products]{Isolated singularities of graphs in warped products and Monge-Ampère equations}
\end{title}
\today
\author{José A. Gálvez}
\address{José A. Gálvez, Departamento de Geometría y Topología,
Universidad de Granada, 18071 Granada, Spain}
 \email{jagalvez@ugr.es}

\author{Asun Jiménez}
\address{Asun Jiménez, Departamento de Geometria, IME,
Universidade Federal Fluminense, 24.020-140  Niterói, Brazil}
 \email{asunjg@vm.uff.br}

\author{Pablo Mira}
\address{Pablo Mira, Departamento de Matemática Aplicada y Estadística, Universidad Politécnica de
Cartagena, 30203 Cartagena, Murcia, Spain.}

\email{pablo.mira@upct.es}

\thanks{The authors were partially supported by
MICINN-FEDER, Grant No. MTM2013-43970-P, Junta de Andalucía Grant
No. FQM325 and Junta de Andalucía, reference P06-FQM-01642.}

%    General info
\subjclass{35J96, 53C42}
%\date{January 1, 1994 and, in revised form, June 22, 1994.}


\keywords{Isolated singularities, convex graphs, Monge-Ampère
equation, warped products, prescribed curvature.}

\begin{abstract}
We study graphs of positive extrinsic curvature with a non-removable
isolated singularity in $3$-dimensional warped product spaces, and
describe their behavior at the singularity in several natural
situations. We use Monge-Ampère equations to give a classification
of the surfaces in $3$-dimensional space forms which are embedded
around a non-removable isolated singularity and have a prescribed,
real analytic, positive extrinsic curvature function at every point.
Specifically, we prove that this space is in one-to-one
correspondence with the space of regular, analytic, strictly convex
Jordan curves in the $2$-dimensional sphere $\S^2$.
\end{abstract}
\maketitle
\section{Introduction}\label{intro}

This paper investigates the geometry of graphs of positive extrinsic
curvature (not necessarily constant) in three-dimensional warped
product spaces $M^2\times_f \R$ around a non-removable isolated
singularity of the graph. The aim here is twofold.

Our first objective is to describe from a purely geometric point of
view the behavior of a graph in the above conditions around an
isolated singularity. For instance, we will show that these graphs,
if bounded, extend continuously with bounded gradient to the
singularity. We will also prove that, in Hadamard manifolds, if the
graph is not bounded around the singularity, then the height
function of the graph has limit $\pm \8$ at the singular point.

Our second objective is to classify the embedded isolated
singularities of prescribed, analytic, positive extrinsic curvature
in $3$-dimensional space forms $\M^3(c)$ of constant curvature $c$.
For that purpose we will regard $\M^3(c)$ as a warped product
manifold and show that the prescribed curvature equation in warped
products is an elliptic equation of Monge-Ampère type. Then, we will
generalize some aspects of our study in \cite{GJM2} of isolated
singularities of Monge-Ampère equations in order to classify the
previous class of embedded isolated singularities in $\M^3(c)$ in
terms of the class of regular, analytic, strictly convex Jordan
curves in $\S^2$.

The first objective above will be carried out in Section
\ref{sec:1}, which can be read independently from the rest of the
paper. Specifically, in Section \ref{sec:1} we will prove the
following three results on the geometry of isolated singularities
for graphs of positive extrinsic curvature in warped product
three-manifolds.
\begin{teorema}\label{thap1}
Let $\Sigma$ be a graph in $\war$ with $\Kext>0$. Assume that
$\Sigma$ has an isolated singularity at $p_0\in M^2$ and is bounded
around $p_0$. Then $\Sigma$ extends across $p_0$ as a continuous
graph and is uniformly non-vertical.
\end{teorema}

\begin{teorema}\label{thap2}
Let $\Sigma$ be a graph with $\Kext>0$ in the Riemannian product
space $M^2\times \R$ (i.e. $f=1$). Assume that $\Sigma$ has an
isolated singularity at $p_0\in M^2$. Then $\Sigma$ extends across
$p_0$ as a continuous graph and is uniformly non-vertical.
\end{teorema}

\begin{teorema}\label{thap3}
Let $\Sigma$ be a graph with $\Kext>0$ in a warped product space
$\war$ which is a Hadamard manifold. Assume that $\Sigma$ has an
isolated singularity at $p_0\in M^2$ and that $\Sigma$ is \emph{not}
bounded around $p_0$. Then, if $z$ denotes the height function of
$\Sigma$, $$\lim_{p\to p_0} z(p)= \pm \8, \hspace{1cm} p\in M^2.$$
\end{teorema}

In Sections \ref{newsec3} and \ref{secgrad}, as a preparation for
our main result in Section \ref{sec:2}, we will study the behavior
of solutions to the general elliptic equation of Monge-Ampère type

 \begin{equation}\label{eq00}
{\rm det }( D^2 z + \mathcal{A}(x,y,z,Dz)) =\varphi(x,y,z,D z)>0
 \end{equation}
around a non-removable isolated singularity, where $\cA(x,y,z,Dz)\in
\cM_2 (\R)$ is a symmetric matrix. Again, these two sections can be
read independently from the rest of the paper. We will show that
when $\cA, \varphi$ are real analytic, the graph $z=z(x,y)$ of any
solution to \eqref{eq00} can be analytically parametrized as an
embedding defined on an annulus, and so that this parametrization
extends analytically to the boundary circle of the annulus which the
parametrization collapses to the singularity (Lemma
\ref{extanalitica}). Also, we will show (Theorem \ref{regama}) that
the limit gradient at the singularity is a regular analytic Jordan
curve in $\R^2$ whose curvature does not change sign. These results
extend to equation \eqref{eq00} some facts proved in \cite{GJM2} for
the \emph{pure} Monge-Ampère equation, i.e. for the case $\cA=0$.
Other results on isolated singularities of elliptic Monge-Ampère
equations can be found in
\cite{ACG,Bey1,Bey2,CaLi,GMM,GJM2,JiXi,Jor,ScWa}.

In Section \ref{sec:2} we give the classification of the surfaces in
a space form $\M^3(c)$ that are embedded around an isolated
singularity, and whose extrinsic curvature at each point is
predetermined by a given positive, analytic function $\cK$ defined
on a neighborhood of the singular point in $\R^3$ (see Theorem
\ref{teo: embedded}). Let us note that, by the Gauss equation
\beq\label{ecgauss} K_G =K_{\rm ext}+c,\eeq prescribing the
extrinsic curvature $K_{\ext}$ on a surface in a space form
$\M^3(c)$ of constant curvature $c$ is equivalent to prescribing the
(intrinsic) Gaussian curvature $K_G$ of the surface.

Theorem \ref{teo: embedded} extends the analogous classification
result proved by the authors in \cite[Theorem 4]{GJM2} for the case
$\M^3(c)=\R^3$. However, the extension of the classification to
$\H^3$ and $\S^3$ is non-trivial. Indeed, while in $\R^3$ the
extrinsic curvature $K_{\text{ext}}$ of a graph $z=z(x,y)$ is given
by the \emph{pure} Monge-Ampère equation \beq\label{pure}
z_{xx}z_{yy}-z_{xy}^2=\varphi(x,y,z,z_x,z_y),\eeq where
$\varphi(x,y,z,z_x,z_y)=K_{\text{ext}}(1+z_x^2+z_y^2)^2$, in $\H^3$
and $\S^3$ the prescribed extrinsic curvature equation has the
general form \eqref{eq00}, i.e. $\cA\neq 0$. This makes our analytic
study of isolated singularities of Monge-Ampère equations in
\cite{GJM2} insufficient for our geometric purposes here, and
justifies the need of the analytic results of Sections \ref{newsec3}
and \ref{secgrad}. Theorem \ref{teo: embedded} also extends results
proved in \cite{GHM,GaMi} for surfaces of constant curvature.

\section{Isolated singularities of graphs in warped products}\label{sec:1}

We devote this section to the proof of Theorems \ref{thap1},
\ref{thap2} and \ref{thap3} stated in the introduction.

Given a Riemannian surface $(M^2,g)$ and a smooth function
$f:\R\flecha (0,\8)$, we define the \emph{three-dimensional warped
product} $M^2\times_f \R$ as the Riemannian manifold $ (M^2\times
\R,\esiz,\esde)$, where
$$\esiz,\esde=f(t)g+dt^2. $$ A surface $\Sigma$ in $M^2\times_f \R$ is a
\emph{graph} if $\pi:\Sigma\flecha \pi (\Sigma)$ is a
diffeomorphism, where $\pi$ stands for the projection $M^2\times_f
\R\flecha M^2$. If we choose coordinates $(x,y)$ on a domain of
$M^2$ that contains $\pi (\Sigma)$, then the graph $\Sigma$ in the
coordinates $(x,y,t)$ is given by $t=z(x,y)$, where $z(x,y)$ is a
smooth function. We will call $z$ the \emph{height function} of the
graph.

\begin{definicion}\label{sing}
Let $\Sigma$ be a smooth graph in $\war$ over a punctured disk
$\cD^*\subset M^2$ around some $p_0\in M^2$. If $\Sigma$ does not
extend as a $C^1$ graph to $\cD^*\cup \{p_0\}$, we will call $p_0$
an \emph{isolated singularity} of $\Sigma$.
\end{definicion}

In what follows we will denote by $\Kext$  the \emph{extrinsic
curvature function} of the oriented graph $\Sigma$, i.e. the
determinant of the second fundamental form $II$ of $\Sigma$ with
respect to its first fundamental form: $K_{\ext}= {\rm det} (II) /
{\rm det} (I)$. Then, the condition $\Kext
>0$ is equivalent to the property that $II$ is (positive or
negative) definite at every point.

Given a graph $\Sigma\subset \war$, we can orient it so that its
unit normal $N$ satisfies $\esiz N,\parc_t\esde \in (0,1]$, where
$t$ is the vertical coordinate in $\war$. We will call $\nu:=\esiz
N,\parc_t\esde$ the \emph{angle function} associated to $\Sigma$,
and we will say that $\Sigma$ is \emph{uniformly non-vertical} if
$\nu\geq c>0$ in $\Sigma$.

As a preparation for the proofs of Theorems \ref{thap1}, \ref{thap2}
and \ref{thap3}, we consider next a special type of local
coordinates $(u,v,t)$ on a warped product space $M^2\times_f \R$.

Let $(M^2,g)$ be an oriented Riemannian surface and $p_0\in M^2$.
For a fixed unit vector $\xi\in T_{p_0} M$, let $\gamma(v)$ be the
unique geodesic in $M^2$ with initial conditions $\gamma(0)=p_0$,
$\gamma'(0)= \xi$. Let $\exp$ and $J$ denote, respectively, the
exponential map and the complex structure of $(M^2,g)$. Then, for
$\ep>0$ small enough, the map $$(u,v)\longmapsto \exp_{\gamma(v)} (u
J\gamma'(v))$$ defines a diffeomorphism from $R_{\ep} :=
(-\ep,\ep)\times (-\ep,\ep)$ into a neighborhood $U\subset M^2$ of
$p_0$, such that the metric $g$ is expressed with respect to $(u,v)$
as $$g= du^2 + G(u,v) dv^2$$ for a positive smooth function $G(u,v)$
in $R_{\ep}$ with $G(0,v)=1$ for all $v\in (-\ep,\ep)$.

Observe that in these coordinates, each curve $v={\rm const.}$ in
$R_{\ep}$  corresponds to a geodesic of $(M^2,g)$.

Let now $\Sigma$ be a graph in $\war$ with an isolated singularity
at $p_0$. If we parametrize $U\times \R$ in terms of the $(u,v,t)$
coordinates defined above, then $\Sigma$ is written in a
neighborhood of $p_0$ as $$\Sigma=\{ (u,v,z(u,v)) : (u,v)\in
\Omega\subset R_{\ep}\setminus \{(0,0)\}\}$$ for some smooth
function $z$ defined on a punctured disk $\Omega$ centered at the
origin.

Consider $\bar{\parc}_u, \bar{\parc}_v$  the tangent coordinate
frame in $\Sigma$ with respect to $(u,v)$, i.e.
$$\bar{\parc}_u = \parc_u + z_u \parc_t,\hspace{1cm}
\bar{\parc}_v=\parc_v + z_v \parc_t,$$ and
 let $\eta$ be the (non-unit) upwards pointing normal vector field to
$\Sigma$ $$\eta:= -z_u \parc_u - \frac{z_v}{G}\parc_v + f \parc_t.$$
Then, bearing in mind that the Levi-Civita connection $\nabla$ in
$\war$ in the coordinates $(u,v,t)$ satisfies $$\nabla_{\parc_u}
\parc_u = -\frac{f'(t)}{2} \parc_t, \hspace{1cm} \nabla_{\parc_u}
\parc_t =\frac{f'(t)}{2 f(t)} \parc_u, \hspace{1cm} \nabla_{\parc_t}
\parc_t =0,$$ a simple computation shows that the second fundamental
form $II$ of $\Sigma$ verifies

$$\def\arraystretch{1.5}\begin{array}{lll}
II(\bar{\parc}_u,\bar{\parc}_u) &=& \esiz \nabla_{\bar{\parc}_u}
\bar{\parc}_u ,\frac{\eta}{||\eta||}\esde \\ & = &
\frac{1}{||\eta||} \left(-f'(z) z_u^2 + f(z) \left(z_{uu} -
\frac{f'(z)}{2}\right)\right).\end{array}$$ If we assume now that
$\Kext
>0$ for $\Sigma$, then
 \begin{equation}\label{eqa1}
-f'(z) z_u^2 + f(z) \left(z_{uu}- \frac{f'(z)}{2} \right) \neq 0
 \end{equation}
for every $(u,v)\in \Omega$. Alternatively, we can rewrite
\eqref{eqa1} as
 \begin{equation}\label{eqa2}
\frac{\parc}{\parc u} \left(\frac{z_u}{f(z)}\right)>\frac{f'(z)}{2
f(z)} \hspace{0.5cm} \text{or else} \hspace{0.5cm}
\frac{\parc}{\parc u} \left(\frac{z_u}{f(z)}\right)<\frac{f'(z)}{2
f(z)}.
\end{equation}

\vspace{0.2cm}

\emph{Proof of Theorem \ref{thap1}:} We assume, for instance, that
the first inequality in \eqref{eqa2} holds (the argument is similar
with the second inequality). As $z$ is bounded by hypothesis, there
is some $c_0\in \R$ such that, for $(u,v)\in \Omega\subset
R_{\ep}\setminus\{(0,0)\}$, $$\frac{\parc}{\parc u}
\left(\frac{z_u}{f(z)}\right)>\frac{f'(z)}{2 f(z)} \geq c_0,$$ and
therefore $$\frac{\parc}{\parc u} \left(\frac{z_u}{f(z)} -c_0
u\right)>0.$$ This condition easily implies that
$$\frac{z_u}{f(z)}-c_0 u$$ is bounded in $\Omega$, from where $z_u$
is also bounded in $\Omega$.

Recall now that, by construction, the coordinates $(u,v)$ depend on
the arbitrary unit vector $\xi$ which determines the $\parc_v$
direction at $(u,v)=(0,0)$. So, a different choice $\hat{\xi}$ of
the vector $\xi$ will result in new coordinates $(\hat{u},\hat{v})$
for which $z_{\hat{u}}$ will be bounded. By choosing $\hat{\xi}$ so
that $\{\parc_u,\parc_{\hat{u}}\}$ are linearly independent at
$(0,0)$, it is easy to deduce then that $z_v$ is also bounded around
$(0,0)$.

On the other hand, a computation shows that the angle function
$\nu=\langle N,\partial t\rangle$ of $\Sigma$ satisfies
\begin{equation}\label{eqa3}
\nu^2 = \esiz N,\parc_t\esde^2 = \frac{1}{||\eta||^2} \esiz
\eta,\parc_t\esde^2 = \frac{f(z)}{f(z) + z_u^2 + \frac{z_v^2}{G}}.
\end{equation}
As $z,z_u$ and $z_v$ are bounded around $(0,0)$, we conclude from
\eqref{eqa3} that $\nu^2\geq c>0$ around $(0,0)$, i.e. $\Sigma$ is
uniformly non-vertical around the isolated singularity. Finally, as
$z_u$ and $z_v$ are bounded in a punctured neighborhood of the
origin we deduce in a standard way that $z$ is continuous at
$(0,0)$.
 This completes the proof of Theorem \ref{thap1}.
$\hfill\Box$

\vspace{0.2cm}

\emph{Proof of Theorem \ref{thap2}:} By the condition $f=1$, we get
from \eqref{eqa1} that $z_{uu}$ has a constant sign on the punctured
disk $\Omega\subset R_{\ep}\setminus \{(0,0)\}$. Thus $z_u$ is
bounded in $\Omega$. The rest of the argument is identical to the
one used in the proof of Theorem \ref{thap1}. $\hfill\Box$

\vspace{0.3cm}

\emph{Proof of Theorem \ref{thap3}:}

\begin{observacion}\label{nota}
Theorem \ref{thap3} is also true if, instead of assuming that $\war$
is a Hadamard manifold, we ask for the following property ${\bf
(P)}$ to hold for some strongly convex geodesic disk $D_r\subset
M^2$ of radius $r$ centered at $p_0$.

\begin{quote}
\emph{Property {\bf (P)}:} For any two points in $D_r\times
\R\subset \war$ there is a unique minimizing geodesic arc joining
both points, and moreover, this geodesic arc is totally contained in
$D_r\times \R$.
\end{quote}

Observe that any Hadamard manifold $\war$ verifies property ${\bf
(P)}$. Basic examples of this type of manifolds are $\R^3$, $\H^3$
and $\H^2\times \R$.

\end{observacion}

 Let $D_r\subset M^2$ be a
strongly convex geodesic disk of radius $r$ centered at $p_0$, that
is, for every two points $p_1,p_2\in\overline{D_r}$ there exists a
unique minimizing geodesic $\gamma$ joining $p_1$ to $p_2$   such
that $\text{int}(\gamma)\subset D_r$. Assume that property ${\bf
(P)}$ holds in $D_r\times \R\subset
M^2\times_f \R$. We will assume without loss of generality that the graph $\Sigma$
is well defined on $D_r\setminus \{p_0\}$, and the second
fundamental form of $\Sigma$ is positive definite for the upwards
pointing unit normal $N$ of $\Sigma$ (recall that $\Kext>0$ on
$\Sigma$). This implies, using that the vertical planes
$\gamma\times \R$ over a geodesic $\gamma\subset M^2$ are totally
geodesic surfaces in $M^2\times_f \R$, that any geodesic
$\widetilde{\gamma}$ in $M^2\times_f \R$ which is tangent to
$\Sigma$ at some point $p\in \Sigma$ lies below $\Sigma$ around $p$
(note that $\widetilde{\gamma}\subset \gamma\times \R$ for some
geodesic $\gamma$ of $M^2$).

We define the \emph{epigraph} of $z$ by $${\rm epi} (z)=\{(p,t)\in
D_r\times \R : t\geq z(p), p \neq p_0\}.$$

 \textsc{Claim}. \label{claim1}
 {\it $\overline{{\rm epi} (z)}$ is a convex subset
of $\overline{D_r}\times \R$.}

Observe that the convexity notion makes sense since we are assuming
that property ${\bf (P)}$ holds in $D_r\times \R\subset M^2\times_f
\R$.

To prove this claim, we take two points $(p_1,t_1)$, $(p_2,t_2)$ in
$\overline{{\rm epi} (z)}\subset \overline{D_r}\times \R$ and prove
that the unique geodesic $\Gamma$ in $\overline{D_r}\times \R$
joining them is contained in $\overline{\rm{epi} (z)}$. We
distinguish several cases.

\emph{Case 1: the geodesic $\gamma$ joining $p_1$ and $p_2$
($p_1\neq p_2$) in $\overline{D_r}$ does not pass through $p_0$.}

First note that if $p_1=p_2 (\neq p_0)$, the geodesic $\Gamma$
corresponds to a vertical segment, so the property holds.

Consider the totally geodesic plane over $\gamma$, that is,
$\gamma\times \R\subset D_r\times \R$. Then, the geodesic $\Gamma$
is contained in $\gamma\times \R$. Let $\alfa=(\gamma\times \R)\cap
\Sigma$.

Let $\theta_0$ be the angle that the geodesic $\Gamma\subset
\gamma\times \R$ makes with the vertical direction $\parc_t$ at the
point $(p_1,t_1)$, and consider the family
$\{\Gamma_{\theta}\}_{\theta\in[0,\theta_0]}$ of geodesics in
$\gamma\times \R$ starting at $(p_1,t_1)$ and making an angle
$\theta$ with $\parc_t$ at this initial point. Note that
$\Gamma_{0}$ is $\{p_1\}\times [t_1,\8)$, that the interior of
$\Gamma_0$ does not intersect $\alfa$, that
$\Gamma_{\theta_0}=\Gamma$, and that all such geodesics only
intersect at the initial point $(p_1,t_1)$.

Once here, observe that the existence of a point $q\in \Gamma$ not
lying in $\overline{{\rm epi} (z)}$ would mean that $\alfa$ is
\emph{above} $\Gamma$ around $q$ in $\gamma\times \R$. But that
would mean that some geodesic $\Gamma_{\theta}$ lies above $\alfa$
in $\gamma\times \R$ and touches $\alfa$ tangentially at some point.
This is a contradiction with the fact that $\Sigma$ has positive
definite second fundamental form for the upwards pointing unit
normal, and hence lies locally above all its tangent geodesics. This
completes the proof of the convexity of $\overline{{\rm epi} (z)}$
in Case 1.

\emph{Case 2: the geodesic $\gamma$ joining $p_1$ and $p_2$ in
$\overline{D_r}$ passes through $p_0$.}

Let $\{x_n\}$ be a sequence of points in $D_r$ with $x_n\to p_2$,
and such that the geodesic in $D_r$ joining $x_n$ and $p_1$ does not
pass through $p_0$. Then, we can also take $(t_n)_n$ such that
$(x_n,t_n)\in \overline{{\rm epi} (z)}$ and $(x_n,t_n)\to (p_2,t_2)$
as $n\to \8$. By Case 1, the geodesic $\Gamma_n$ joining $(p_1,t_1)$
with $(x_n,t_n)$ is contained in $\overline{{\rm epi} (z)}$. Taking
limits, $\Gamma_n$ converge to the geodesic $\Gamma$ joining
$(p_1,t_1)$ with $(p_2,t_2)$. In particular, $\Gamma$ is contained
in $\overline{{\rm epi} (z)}$ as we wanted to prove. We remark that
this argument also holds in the case $p_0=p_2$.

\emph{Case 3: $p_1=p_2=p_0$.}

In this case we take two points $(p_0,t_1)$ and $(p_0,t_2)$ in
$\overline{{\rm epi} (z)}$ with $t_1<t_2$. Take a sequence
$(x_n,t_n)\to (p_0,t_1)$ with $(x_n,t_n)\in {\rm epi}
(z)$ for all $n$. Then, the vertical segments
$$\Gamma_n:= \{x_n\}\times [t_n,t_n+t_2-t_1] \subset \overline{{\rm
epi} (z)}$$ are geodesics, so by taking limits $\Gamma_n\to \Gamma
=\{p_0\}\times [t_1,t_2],$ which is also contained in
$\overline{{\rm epi} (z)}$.

Thus, we have proved that $\overline{{\rm epi} (z)}$ is a convex
subset of $\overline{D_r}\times \R$. That is, we have finished the
proof of the Claim above.

Let us observe now that there is some $t_0\in \R$ such that
$(p_0,t_0)\in {\rm int} (\overline{{\rm epi} (z)})$. For this, let
$p_1,p_2\in D_r$ such that $p_0$ lies in the geodesic $\gamma$
joining $p_1$ with $p_2$. If we take $t_1>z(p_1)$, then $(p_1,t_1)
\in {\rm int} (\overline{{\rm epi} (z)})$ and the geodesic in
$D_r\times \R$ joining $(p_1,t_1)$ with $(p_2,z(p_2))$ passes
through some point of the form $(p_0,t_0)$. By standard convexity
arguments, all the points in such a geodesic arc lie in ${\rm int}
(\overline{{\rm epi} (z)})$ except the endpoints. In particular,
$(p_0,t_0)\in {\rm int} (\overline{{\rm epi} (z)})$.

Consider now $(p_0,t_0)\in {\rm int} (\overline{{\rm epi} (z)})$,
and let $\ep>0$ so that $D_{\ep}\times (t_0-\ep,t_0+\ep)\subset
\overline{{\rm epi} (z)}$, where $D_{\ep}\subset D_r$ is a geodesic
disk of radius $\ep$ centered at $p_0$. Then, $D_{\ep}\times
(t_0-\ep,\8)\subset \overline{{\rm epi} (z)}$ from where it follows
that $(p_0,t_1)\in {\rm int} (\overline{{\rm epi} (z)})$ for all
$t_1>t_0$. By convexity, there are two possibilities:

\begin{enumerate}
  \item[(i)]
$\overline{{\rm epi} (z)}\cap (\{p_0\}\times \R) =
  \{p_0\}\times [h,\8)$ with $\{p_0\}\times (h,\8)\subset {\rm int} (\overline{{\rm epi}
  (z)})$, or
  \item[(ii)] $\overline{{\rm epi} (z)}\cap (\{p_0\}\times \R) =
  \{p_0\}\times \R$, with $\{p_0\}\times \R\subset {\rm int} (\overline{{\rm epi}
  (z)})$.
\end{enumerate}

We will prove first of all that in case (i) above we have $\lim_{p\to p_0}
z(p)=h$. Thus, case (i) cannot happen since we assumed that $z$ is not bounded around $p_0$.

Let $\{p_n\}$ be a sequence of points in $D_r\setminus \{p_0\}$
converging to $p_0$. As we are in case (i), given $\delta >0$ there exists
some $\ep>0$ such that
$$(p_0,h+\delta) \in D_{\ep} \times (h+\delta-\ep,h+\delta+\ep)
\subset \overline{{\rm epi}
  (z)}$$ and $$(p_0,h-\delta)
\in D_{\ep} \times (h-\delta-\ep,h-\delta+\ep) \subset (D_r\times
\R)\setminus \overline{{\rm epi}
  (z)}.$$ In particular, we have $$D_{\ep}\times
  [h+\delta,\8)\subset {\rm int} (\overline{{\rm epi}
  (z)}) \hspace{0.5cm} \text{and} \hspace{0.5cm} D_{\ep}\times
  (-\8,h-\delta]\subset {\rm int}( (D_r\times \R)\setminus \overline{{\rm epi}
  (z)}).$$
Now, observe that, since $(p_n,z(p_n))$ lies in the boundary of
${\rm epi} (z)$, it holds $(p_n,z(p_n))\not\in {\rm int}
(\overline{{\rm epi}
  (z)}) \cup {\rm int}( (D_r\times \R)\setminus \overline{{\rm epi}
  (z)}).$ Thus, choosing $n_0$ so that if $n\geq n_0$ we have
  $p_n\in D_{\ep}$, we can conclude from the previous condition that
  $|z(p_n)-h|<\delta$ for every $n\geq n_0$.
  
 This proves that $\lim_{p\to p_0}
z(p)=h$. Hence case (i) cannot happen, i.e., the condition in case (ii) must hold.

Now, the same argument used above to prove that $\lim_{p\to p_0}
z(p)=h$ in case (i) also shows that $\lim_{p\to p_0}
z(p)=-\8$ in case (ii). Observe that we would have obtained $\lim_{p\to p_0}
z(p)=+\8$ if we had worked with the opposite (downwards pointing) unit normal of $\Sigma$.
This completes the proof of Theorem
\ref{thap3}. $\hfill\Box$

We finish this section with the following result characterizing
isolated singularities of graphs of positive extrinsic curvature in
Hadamard warped products.

\begin{corolario}\label{corap}
Let $\Sigma\subset M^2\times_f \R$ be a graph with $\Kext>0$ over a
punctured disk $D^*\subset M^2$ centered at $p_0\in M^2$, and assume
that $M^2\times_f \R$ is a Hadamard manifold. Then, exactly one of
the following three situations happen:
 \begin{enumerate}
   \item
$\Sigma$ extends as a smooth $C^1$-graph across $p_0$.
 \item
$\Sigma$ extends continuously (but not $C^1$-smoothly) across $p_0$,
and is uniformly non-vertical at $p_0$.
 \item
The height function of $\Sigma$ tends to $+\8$ or to $-\8$ at $p_0$.
In particular, the metric of $\Sigma$ is complete around the
puncture.
 \end{enumerate}
Moreover, assume that the sectional curvature of $M^2\times_f \R$ is
bounded from below on a tubular neighborhood $D_{\ep}\times\R$ of
the geodesic $\{p_0\}\times \R$ by a number $i(p_0)\in \R$. Then the
third situation above cannot happen provided $\Kext>-i(p_0)$.
\end{corolario}
\begin{proof}
By Theorems \ref{thap1} and \ref{thap3} we only need to prove the
final assertion. This is a consequence of the Gauss equation for
$\Sigma$ in $M^2\times_f \R$, which implies that if $\Kext>-i(p_0)$,
the Gaussian curvature of $\Sigma$ is bounded from below by a
positive constant around $p_0$.

But now, observe that in $M^2\times_f \R$, the distance between
$(p_1,t_1)$ and $(p_2,t_2)$ is at least $|t_2-t_1|$. Hence, the fact
that the height function of $\Sigma$ tends to $\pm \8$ around $p_0$
indicates that there exist points $p,q\in\Sigma$ arbitrarily far
away from each other, and also from $\parc \Sigma$. This contradicts that on a surface with
curvature bounded from below by some $c>0$, the length of any
minimizing geodesic is at most $\pi /\sqrt{c}$.
\end{proof}

\begin{ejemplo}
The half-space model of the hyperbolic 3-space
$\H^3=(\{(x_1,x_2,x_3)\in\R^3:\
x_3>0\},\frac{1}{x_3^2}(dx_1^2+dx_2^3+dx_3^2))$ can be written as
the warped product
 \begin{equation}\label{horomodel}
(\R^2\times\R, e^{-2z}(dx^2+dy^2)+dz^2).
 \end{equation}
Hence, $\R^2\times_{e^{-2z}} \R$ is a Hadamard manifold. In this
model, the graphs $z=z(x,y)$ correspond geometrically to graphs in
$\H^3$ over horospheres. For every $R>0$, the graph $z=z(x,y)$ of
the function $$z(x,y)= \log \left(R- \sqrt{R^2-x^2-y^2}\right),
\hspace{1cm} 0<x^2+y^2<R^2 ,$$ is a piece of a horosphere $\Sigma_R$
in $\H^3$. In particular, $\Sigma_R$ has $K_{\rm ext}=1$. Note that
$(0,0)$ is an isolated singularity of the graph, that $z\to -\8$ as
$(x,y)\to (0,0)$, and that $\Sigma_R$ is complete around the
puncture. This provides an example of Case (3) in Corollary
\ref{corap}. $\Sigma_R$ also shows  that the inequality $K_{\rm ext}
>-i(p_0)$ in the last assertion of Corollary \ref{corap} is
necessary, and cannot be weakened to $K_{\rm ext} \geq -i(p_0)$.
\end{ejemplo}

\section{Isolated singularities of the Monge-Ampère equation: preliminaries}\label{newsec3}

Let us consider the general elliptic equation of Monge-Ampère type
in dimension two, which is the following fully nonlinear PDE:
 \begin{equation}\label{eq0}
{\rm det }( D^2 z + \mathcal{A}(x,y,z,Dz)) =\varphi(x,y,z,D z)>0.
 \end{equation}
Here, $Dz,D^2 z$ denote respectively the gradient and the Hessian of
$z$, and $\mathcal{A}(x,y,z,Dz)\in \cM_2(\R)$  is symmetric. The
Monge-Ampère equation \eqref{eq0} can be rewritten as
 \beq\label{eq1}
   Az_{xx}+2B z_{xy}+C z_{yy} +z_{xx} z_{yy}-z_{xy}^2=E,\eeq
where $A=A(x,y,z,z_x,z_y),\dots, E=E(x,y,z,z_x,z_y)$ are defined on
an open set $\cU\subset \R^5$ and satisfy on $\cU$ the ellipticity
condition \beq\label{elliptic} \cD:=AC-B^2+E>0. \eeq

We will study solutions $z$ to the elliptic equation \eqref{eq1}
around a non-removable isolated singularity, and for simplicity we
will assume that this singularity is placed at $(0,0)$. All our
results can be trivially adapted if the singularity is placed
elsewhere.

\vspace{0.2cm}

{\bf Convention:} From now on we will use the following notations:

 \begin{itemize}
   \item
$\Omega=\{(x,y)\in\R^2 : 0<x^2+y^2<\rho^2\} $, a punctured disc
centered at the origin.
 \item
$\cU\subset \R^5$ is an open set.
 \item
$A,\dots, E$ are functions in $C^{2} (\cU)$, which satisfy in
$\cU$ the ellipticity condition \eqref{elliptic}.
 \item
For a function $z\in C^2 (\Omega)$, we define
\begin{equation}\label{hype} \mathfrak{H}:=
\{(x,y,z(x,y),z_x(x,y),z_y(x,y)) : (x,y)\in\Omega\}.
 \end{equation}
 \end{itemize}

\begin{definicion}\label{def: bounded}
We say that a solution $z\in C^2(\Omega)$ to \eqref{eq1} for the
coefficients $A,\dots, E$ is a \emph{singular solution} of
\eqref{eq1} in $\Omega$ if
 \begin{enumerate}
   \item
$z$ is not $C^1$ at the origin.
   \item
$\overline{\mathfrak{H}}$ is a compact subset of $\cU$.
\end{enumerate}
  \end{definicion}

In the case that the coefficients $A,\dots, E$ are real analytic on
$\cU$, the solution $z$ is also real analytic on $\Omega$. Also, it
can be easily proved that in the conditions of the definition,
$z(x,y)$ extends \emph{continuously} to $(0,0)$.

\emph{From now on we will assume for simplicity that any singular
solution to \eqref{eq1} has been continuously extended to the origin
by $z(0,0)=0$, and that $\cH\neq \emptyset$, where}
$$\cH:= \cU\cap \{(x_1,\dots, x_5)\in \R^5 : x_1=x_2=x_3=0\}.$$

Also, we will be using the standard classical notation $p=z_x$,
$q=z_y$, $r=z_{xx}$, $s=z_{xy}$, $t=z_{yy}$.

Let $z\in C^2(\Omega)$ be a singular solution to \eqref{eq1}. It
follows then from the ellipticity condition \eqref{elliptic} that
the expression \beq\label{metrica}
ds^2=\varepsilon\left((z_{xx}+C)dx^2+2(z_{xy}-B)dxdy+(z_{yy}+A)dy^2
\right) \eeq is a Riemannian metric on $\Omega$ for $\varepsilon=1$
or $\varepsilon=-1$. Then, it is a well known fact that $(\Omega,
ds^2)$ admits global conformal parameters $w:=u+iv$ such that

\beq\label{conforme} ds^2=\frac{\sqrt{\cD}}{u_xv_y-u_yv_x}|dw|^2.
\eeq That is, there exists a $C^2$ diffeomorphism

\begin{equation}\label{tristar}
\Phi:\Omega\rightarrow  \Lambda:=\Phi(\Omega)\subset \R^2 ,\qquad
(x,y)\mapsto \Phi(x,y)=(u(x,y),v(x,y))
\end{equation}
satisfying \beq x_uy_v-x_vy_u>0,\eeq and the Beltrami system
 \beq\label{sist}  \left(\begin{array}{c}
 v_x\\v_y
 \end{array}\right)=\frac{1}{\sqrt{\cD}}\left(\begin{array}{cc}s-B & -(C+r)\\A+t & -(s-B)
 \end{array}\right)\left(\begin{array}{c}
 u_x\\u_y\end{array}\right).\eeq

Here, $\Lambda$ is a domain in $\R^2\equiv \C$ which is conformally
equivalent to either the punctured disc $\D^*$ or an annulus
$\A_R=\{z\in \C : 1<|z|<R\}$.

In this situation, motivated by \cite{HeB}, we introduce the
following definition.

\begin{definicion}\label{HeB}
A solution $z$ to \eqref{eq1} in $\Omega$ satisfies the
\emph{Heinz-Beyerstedt condition}, in short \emph{HeB-condition}, if
 $A_p$, $A_q+2B_p$, $C_p+2B_q$ and $C_q$ are Liptschitz continuous in $\overline{\Omega}$ when they
 are considered as functions of $x$ and $y$.
\end{definicion}

It was proved by Heinz and Beyerstedt (cf. \cite[Lemma 3.3]{HeB})
that if $z\in C^2(\Omega)$ is a singular solution to \eqref{eq1}
which satisfies the HeB-condition, then $\Lambda$ is conformally
equivalent to some annulus $\A_R$.

Let us point out that the HeB-condition holds automatically for
singular solutions of a wide class of Monge-Ampère equations, as
explained in the next lemma.

  \begin{lema}\label{lem: estrella}
 Suppose that  the coefficients $A,\ldots,E:\cU\subset\R^5\lflecha\R$ in \eqref{eq1} satisfy the condition:

   $$\begin{array}{lr}
\begin{array}
{l} \mbox{The functions  $A_p$, $A_q+2B_p$, $C_p+2B_q$ and $C_q$  do
not depend on}\\  \mbox{$p$ and $q$ in $\cU$.}
\end{array}
&\qquad  (\star)
   \end{array}$$
       Then any singular solution $z(x,y)$ to \eqref{eq1} satisfies the HeB-condition in Definition \ref{HeB}.
\end{lema}

\begin{proof}
We denote by $F(x,y,z)$ any of the functions in the statement of the
condition $(\star)$. Then, $F$ can be seen as a function $\tilde{F}$
depending on the variables $(x,y)$ using the composition
$\tilde{F}(x,y)=(F\circ G) (x,y)$ where $G(x,y)=(x,y,z(x,y))$. As
$A,\ldots,E\in C^2(\cU)$, we have $F\in C^1(\cU)$ and $\tilde{F}\in C^1(\Omega)\cup C^0 (\overline{\Omega})$. Note that
$$\tilde{F}_x (x,y)= F_x (x,y,z(x,y)) + F_z (x,y,z(x,y)) z_x (x,y),$$ which is bounded in $\Omega$ by condition (2) in Definition \ref{def: bounded}. Analogously, $\tilde{F}_y$ is also bounded in $\Omega$. Hence, $\tilde{F}\in C^{0,1}(\overline{\Omega})$ (see \cite[p. 154]{GiTr}). This concludes the proof of Lemma \ref{lem: estrella}.
\end{proof}

For later use, let us also point out the following basic result.

\begin{lema}\label{lemconvex}
Let $z\in C^2(\Omega)$ be a  singular solution to \eqref{eq1}, where
$\Omega$ is an open domain of $\R^2$ (not necessarily a punctured
disc). Then, for sufficiently large constants $a,c>0$, the function
 \begin{equation}\label{zconvex}
z^*(x,y)= z(x,y)+ \frac{\ep a}{2} x^2 + \frac{\ep c}{2} y^2
 \end{equation}
satisfies that its graph $\{(x,y,z^*(x,y)): (x,y)\in \Omega\}$ is a
locally convex surface in $\R^3$, where $\ep\in \{-1,1\}$ is as in \eqref{metrica}.
\end{lema}
\begin{proof}
We will assume $\ep=1$; the case $\ep=-1$ is analogous. As the coefficients $A, \dots, E$ are bounded on $\mathfrak{H}$, we
can find constants $a,c>0$ such that
$$ c-C>0,\qquad
a-A>0,\qquad (c-C)(a-A)-B^2>0,
$$ i.e.  the matrix \beq\label{N} \mathcal{N}=\left(
                                                                         \begin{array}{cc}
                                                                           c-C &B \\
                                                                           B & a-A \\
                                                                         \end{array}
                                                                       \right)
\eeq is positive definite. Now, we may use the fact that $ds^2$ in
\eqref{metrica} is positive definite to conclude that the symmetric
bilinear form
$$ds^2+(dx,dy)\mathcal{N}(dx,dy)^T=(dx,dy)\left(
                                                            \begin{array}{cc}
                                                              r+c & s \\
                                                              s & t+a \\
                                                            \end{array}
                                                          \right)
(dx,dy)^T$$ is also positive definite, that is, it is a Riemannian
metric on $\Omega$.

On the other hand, a straightforward computation shows that the
matrix of the second fundamental form of the graph of $z^*(x,y)$ in
\eqref{zconvex} with respect to the upwards-pointing unit normal is given by
$$
{\rm II}^*=\frac{1}{\sqrt{1+(p+cx)^2+(q+ay)^2}}\left(
                              \begin{array}{cc}
                                r+c & s\\
                                s & t+a \\
                              \end{array}
                            \right),$$
which we have just proved is positive definite on $\Omega$. In
particular, the graph of $z^*(x,y)$ has positive curvature at every
point, which proves the assertion.
\end{proof}

\section{The limit gradient of singular solutions}\label{secgrad}

Let $z(x,y)$ be a singular solution to \eqref{eq1} that satisfies
the HeB-condition, and let $(u,v)$ be the conformal parameters
associated to $z(x,y)$ introduced in the previous section. As
explained there, $(u,v)$ are defined on a domain $\Lambda\subset \C$
conformally equivalent to an annulus. Thus, we may assume $\Lambda$
to be a quotient strip $\Gamma_R:=\{w\in \C: 0<\Im w < R\} /(2\pi
\Z)$. From now on $(u,v)$ will denote the canonical coordinates in
this strip, and the quantities $x,y,z,p,q$ associated to the
solution $z(x,y)$ will be seen as functions depending on $(u,v)$.
Note that all of them are $2\pi$-periodic in the variable $u$, by
construction.

Let $G=\{(x,y,z(x,y)):(x,y)\in \Omega\}\subset \R^3$ be the graph of
$z(x,y)$. Then we can parameterize $G$ as

 \begin{equation}\label{grafconf}
\psi (u,v)=(x(u,v),y(u,v),z(u,v)):\Gamma_R\flecha G\subset \R^3,
  \end{equation}
so that $\psi$ extends continuously to $\R$ with
$\psi(u,0)=(0,0,0)$. We next prove a boundary regularity result for
the map
 \begin{equation}\label{zneg}
 {\bf z} (u,v)= (x(u,v),y(u,v),z(u,v),p(u,v),q(u,v)):\Gamma_R\flecha \R^5.
 \end{equation}

\begin{lema}\label{extanalitica}
Let $z\in C^2(\Omega)$ be a singular solution to \eqref{eq1} that
satisfies the HeB-condition, and assume that $A,\ B,\ C,\ E \in
C^k(\cU)$, $k\geq 2$ (resp. $A,\ B,\ C,\ E \in C^{\omega} (\cU)$).
Let $(u,v)$ be the conformal coordinates in $\Gamma_R$ associated to
$z$ as explained previously. Then the map ${\bf z}(u,v)$ in
\eqref{zneg} extends as a $C^{k,\alpha}$ map $\forall
\alpha\in(0,1)$ (resp. as a real analytic map) to $\Gamma_R\cup \R$.
\end{lema}

\begin{proof}[Proof of Lemma \ref{extanalitica}]
The proof in an adaptation of Claim 1 in \cite{GJM2} to a more
general context.

To start with, we follow a bootstrapping method. Consider an arbitrary
point of $\R$, which we will suppose without loss of generality to
be the origin. Also, consider for $0<\delta<R$ the domain
$\D^+=\{(u,v): 0<u^2+v^2<\delta^2, v>0  \}$.

From  \eqref{sist} it follows that

 \beq \label{d1}\begin{array}{l}
  p_u=\sqrt{ \cD }y_v+ B y_u- C x_u,\\
  p_v=-\sqrt{\cD}y_u+ B y_v- C x_v,\\
  q_u=-\sqrt{ \cD }x_v+ B x_u- A y_u,\\
  q _v=\sqrt{ \cD }x_u+ B x_v- A y_v.\\
  \end{array} \eeq
And by derivation in \eqref{d1} we obtain (cf. \cite{HeB})
\beq\label{laplacianos}\begin{array}{lcl}
\Delta x&=&h_1(x_u^2+x_v^2)+h_2(x_uy_u+x_vy_v)+h_3(y_u^2+y_v^2)+h_4(x_uy_v-x_vy_u),\\
\Delta
y&=&\tilde{h}_1(x_u^2+x_v^2)+\tilde{h}_2(x_uy_u+x_vy_v)+\tilde{h}_3(y_u^2+y_v^2)+\tilde{h}_4(x_uy_v-x_vy_u),\end{array}\eeq
where the coefficients   $h_1=h_1(x,y,z,p,q),\ldots
,\tilde{h}_4=\tilde{h}_4(x,y,z,p,q)$ are
\beq\label{coefhi} \begin{array}{lcl}h_1&=&B_q-\frac{1}{2\cD}(\cD_x+\cD_z p-\cD_p C+\cD_qB),\\
h_2&=&-A_q-B_p-\frac{1}{2\cD}(\cD_y+\cD_zq+\cD_pB-\cD_qA),\\
h_3&=&A_p,\\
h_4&=&\frac{1}{\sqrt{\cD}}(A_x+B_y+A_zp+B_zq-A_pC+(A_q+B_p)B-B_qA-\frac{1}{2}\cD_p),\\
\tilde{h}_1&=&C_q,\\
\tilde{h}_2&=&-B_q-C_p-\frac{1}{2\cD}(\cD_x+\cD_z p- \cD_p C+\cD_qB),\\
\tilde{h}_3&=&B_p-\frac{1}{2\cD}(\cD_y+\cD_z q+\cD_pB-\cD_qA),\\
\tilde{h}_4&=&\frac{1}{\sqrt{\cD}}(C_y+B_x+C_z q+B_z
p-B_pC+(B_q+C_p)B-C_q A-\frac{1}{2}\cD_q),
\end{array}\eeq
all of them evaluated at ${\textbf z}(u,v)=(x,y,z,p,q)(u,v)$.
Besides, from the Cauchy-Scwharz inequality we get $|x_uy_v-x_v
y_u|\leq \frac{1}{2}(|D x|^2+|D y|^2)$ and $|x_uy_u+x_v y_v|\leq
\frac{1}{2}(|D x|^2+|D y|^2)$.

Hence, letting $Y=(x,y):\D^+\longrightarrow\Omega$, we get from
\eqref{laplacianos} and the fact that $h_1,\ldots ,\tilde{h}_4$ are
bounded (since $\overline{\mathfrak{H}}$ is a compact subset of
$\mathcal{U}$) that \beq\label{ineqlaplac}|\Delta Y|\leq c(|D
x|^2+|D y|^2) \eeq for some constant $c>0$.

Note that $Y\in C^2(\D^+)\cap C^0(\overline{\D^+})$ with
$Y(u,0)=(0,0)$ for all $u$. Hence, we can deduce from Heinz's
Theorem in \cite{He} that $Y\in
C^{1,\alpha}(\overline{\D^+_{\varepsilon}})$  for all
$\alpha\in(0,1)$, where $\D^+_{\varepsilon}=\D^+\cap
B(0,\varepsilon)  $ for some $0<\varepsilon<\delta$.

As the right hand side terms in \eqref{d1} are bounded in
$\overline{\D^+_{\varepsilon}}$, then $p,q\in
W^{1,\infty}(\overline{\D^+_{\varepsilon}})$. Therefore, $p,q\in
C^{0,1}(\overline{\D^+_{\varepsilon}})$ (see \cite[p. 154]{GiTr}).

Once here, noting that \beq\label{zfacil0}
 z_u=px_u+qy_u, \qquad
 z_v=px_v+qy_v,
 \eeq
we obtain $z\in C^{1,\alpha}(\overline{\D^+_{\varepsilon}})$
$\forall \alpha\in (0,1)$. Thus, the right hand side functions in
\eqref{d1} are Hölder continuous of any order in
$\overline{\D^+_{\varepsilon}}$, and so $p,q\in
C^{1,\alpha}(\overline{\D^+_{\varepsilon}})$ $\forall
\alpha\in(0,1)$.

With this, we get from \eqref{laplacianos}  that $\Delta Y$ is
Hölder continuous in $\overline{\D^+_{\varepsilon}}$. A standard
potential analysis argument (see \cite[Lemma 4.10]{GiTr}) ensures
that $x,y\in C^{2,\alpha}(\overline{\D^+_{\varepsilon/2}})$, and by
formula \eqref{d1} we have $p,q\in
C^{2,\alpha}(\overline{\D^+_{\varepsilon/2}})$. So, from
\eqref{zfacil0}, $z\in
C^{2,\alpha}(\overline{\D^+_{\varepsilon/2}})$.

By applying the same argument to $Y_u$ and  $Y_v$, we obtain that
$x,y,z,p,q\in C^{3,\alpha}(\overline{\D^+_{\varepsilon/4}})$. A
recursive process then shows that ${\textbf z}=(x,y,z,p,q)$ is
$C^{k,\alpha}$ $\forall \alpha\in(0,1)$ (resp. $C^{\8}$) at the
origin. As we can do the same argument for all points of $\R$ and
not just the origin, we conclude that ${\textbf z}(u,v)\in
C^{k,\alpha} (\Gamma_R\cup \R)$ (resp.  ${\textbf z}(u,v)\in C^{\8}
(\Gamma_R\cup \R)$).

Finally suppose that  $ A, B, C, E$ are analytic. A computation in
the same spirit of formula \eqref{laplacianos} shows that the
Laplacians of $z,p,q$ are given by: \beq\label{laplacianos2}
\def\arraystretch{1}\begin{array}{lll}
\Delta p&=&(\sqrt{\cD\circ {\textbf z}})_uy_v-(\sqrt{\cD\circ {\textbf z}})_vy_u+(B\circ {\textbf z})_uy_u+(B\circ {\textbf z})_vy_v\\
& & +(B\circ {\textbf z})\Delta y-(C\circ {\textbf z})\Delta
x-(C\circ {\textbf
z})_ux_u-(C\circ {\textbf z})_vx_v,\\ & & \\

\Delta q&=&-(\sqrt{\cD\circ {\textbf z}})_ux_v+(\sqrt{\cD\circ
{\textbf z}})_vx_u+(B\circ {\textbf z})_u x_u+(B\circ {\textbf z})_v
x_v \\ & &+(B\circ {\textbf z})\Delta x-(A\circ {\textbf z})\Delta
y-(A\circ
{\textbf z})_uy_u-(A\circ {\textbf z})_v y_v, \\ & & \\
\Delta z&=&p_ux_u+p_vx_v+q_uy_u+q_vy_v+p\Delta x+q \Delta y.\\

\end{array}\eeq

Therefore, $\textbf{z} (u,v)$ satisfies \beq\label{mixto} \Delta
\textbf{z} = h(\textbf{z},\textbf{z}_u,\textbf{z}_v)\eeq where
$h:\mathcal{O}\subset \R^{15}\flecha \R^5$ is a real analytic
function on an open set $\mathcal{O}$ of $\R^{15}$ containing the
closure of the bounded set $\{({\textbf z},{\textbf z}_u,{\textbf
z}_v)(u,v): (u,v)\in \Gamma_R\}$. Moreover, noting that
$$\textbf{z} (u,v)=(\psi(u,v),\phi(u,v)):\Gamma_R\flecha \R^3\times \R^2\equiv \R^5$$ where $\psi(u,v)=(x(u,v),y(u,v),z (u,v))$   and $\phi(u,v)=(p(u,v),q(u,v))$, we find that
${\textbf z}(u,v)$ is a solution to \eqref{mixto} that satisfies the
mixed initial conditions

$$\left\{ \def\arraystretch{1.5} \begin{array}{l} \psi(u,0)=(0,0,0), \\ \phi_v(u,0)^T=\left(
                                                \begin{array}{ccc}
                                                  -C(0,0,0,\phi(u,0)) & B(0,0,0,\phi(u,0)) & 0 \\
                                                  B (0,0,0,\phi(u,0))& -A (0,0,0,\phi(u,0))& 0 \\
                                                \end{array}
                                              \right)\psi_v(u,0)^T.\end{array}\right.$$
As we have shown that ${\textbf z}\in C^{\8}(\Gamma_R\cup \R)$, we
are in the conditions to apply Theorem 3 in \cite{Mu} to
$\textbf{z}$ around every point in $\R$, and we can conclude that
$\textbf{z}$ is real analytic in $\Gamma_R\cup\R$. \end{proof}

\begin{definicion}
Let $z(x,y)$ be a singular solution to \eqref{eq1}. We define the
\emph{limit gradient} of $z$ at the origin to be the set
$\gamma\subset \R^2$ of points $\xi\in \R^2$ for which there is a
sequence $q_n\to (0,0)$ in $\Omega$ such that $(z_x,z_y)(q_n)\to
\xi$.
\end{definicion}

The following theorem is the main result of this section, and
describes the geometry of the limit gradient of a singular solution
to \eqref{eq1} that satisfies the HeB-condition, when the
coefficients of \eqref{eq1} are real analytic.

\begin{teorema}\label{regama}
Let $z(x,y)$ be a singular solution to \eqref{eq1} that satisfies
the HeB-condition, and assume that the coefficients $A,\ B,\ C,\ E$
of \eqref{eq1} are real analytic. Let $\gamma$ denote the limit
gradient of $z$ at the origin.

Then, $\gamma$ is a regular, convex, real analytic Jordan curve in
$\mathcal{H}\subset\R^2$.
\end{teorema}
\begin{proof}
Let $\gamma\subset \R^2$ denote the limit gradient of $z(x,y)$. By
Lemma \ref{extanalitica} we can extend the map $(p(u,v),q(u,v))$
analytically to $\Gamma_R\cup \R$, and so we may consider the
$2\pi$-periodic map $(\alfa(u),\beta(u)):=(p(u,0),q(u,0))$. As it is
clear that $\gamma=\{(\alfa(u),\beta(u)):u\in \R\}$, we can conclude
that $\gamma$ is a closed curve in $\R^2$, possibly with
singularities. In particular, we may parameterize $\gamma$ in an
analytic, $2\pi$-periodic way, as $\gamma(u)=(\alfa(u),\beta(u))$.
\begin{claim}
\label{gammarega} $\gamma'(u)$ vanishes, at most, at four points in
$[0,2\pi)$.
\end{claim}

\begin{proof}[Proof of Claim \ref{gammarega}]

By contradiction, assume that $\gamma'(u_i)=(0,0)$ for $u_i\in
[0,2\pi)$, $i=1,\dots, 5$. Observe that $p_u (u_i,0)= q_u (u_i,0)=0$
for $i=1,\dots, 5$. Then, since $x(u,0)=y(u,0)=0$ for every $u\in
\R$, by \eqref{d1} we see that $Dx (u_i,0)=Dy(u_i,0)=0$, $i=1,\dots,
5$. Consider now the Taylor series expansions of $x(u,v)$ and
$y(u,v)$ at one of the singular points $(u_i,0)$, and denote by
$P_x(u,v)$ and $P_y(u,v)$ their respective lower order terms. Note
that both $P_x(u,v)$ and $P_y(u,v)$ are homogeneous polynomial of
degree at least two.

Next, recall that $x(u,v)$ and $y(u,v)$ satisfy the elliptic system
of PDEs \eqref{laplacianos}. By comparing power series expansions in
this equation, it is easily seen that the  polynomial of least
degree among $P_x(u,v)$ and  $P_y(u,v)$ must be harmonic (if they
have the same degree, both of them are harmonic). Assume momentarily
that it is $P_x(u,v)$. Hence the first term in the power series
expansion of $x(u,v)$ at $(u_i,0)$ is a harmonic polynomial of
degree $\geq 2$. In these conditions, it is well-known that the
$v=0$ axis (which is a nodal curve of $x(u,v)$) is crossed at
$(u_i,0)$ by at least another nodal curve of $x(u,v)$. If the least
degree polynomial were $P_y(u,v)$, the same could be said about
$y(u,v)$.

As there are five points $(u_i,0)$ by hypothesis, the previous
argument shows the following fact: for at least one of
$f(u,v)=x(u,v)$ or $f(u,v)=y(u,v)$, the $v=0$ axis is crossed
transversely by some nodal curve of $f(u,v)$ at three or more points
$(u_i,0)$.

Assume for the moment that $f(u,v)=x(u,v)$; if $f(u,v)=y(u,v)$ the
argument is analogous.

By construction, the map \eqref{grafconf} is a diffeomorphism from
$\Gamma_R =\Sigma_R /(2\pi \Z)$ into the graph $G=\{(x,y,z(x,y)):
(x,y)\in \Omega\}$. As $G$ is a graph, $G\cap \{x=0\}\subset \R^3$
is formed by exactly two regular curves with an endpoint at the
singularity $(0,0,0)$. In other words, the function $x(u,v)$ only
has two nodal curves in $\Gamma_R$, which is a contradiction. This
shows that $\gamma'(u)$ vanishes at most at four points in
$[0,2\pi)$. This finishes the proof of Claim \ref{gammarega}.
\end{proof}

\begin{claim}
\label{gammacon} $\gamma(\R)\subset \R^2$ bounds a compact, strictly
convex set of $\R^2$.
\end{claim}
\begin{proof}[Proof of Claim \ref{gammacon}]
Consider $z^*(x,y)$ as in \eqref{zconvex}, where the constants
$a,c>0$ satisfy the conditions of Lemma \ref{lemconvex}. Hence,
$z^*$ has positive Hessian at every point and an isolated
singularity at the origin whose limit gradient is $\gamma$. For simplicity, assume for the proof of this Claim that $\ep=1$; the case $\ep=-1$ can be proved
analogously.

A parametrization of the graph of $z^*(x,y)$ is given by
$\psi^*:\Gamma_R\flecha \R^3$,
$$\psi^*(u,v)= \left(x(u,v),y(u,v),z(u,v)+\frac{c}{2}
x(u,v)^2+\frac{a}{2} y(u,v)^2\right).$$ The upwards-pointing unit
normal of $\psi^*$ in $\Gamma_R$ is given by
$$N^*(u,v)=\frac{1}{\sqrt{1+(p+cx)^2+(q+ay)^2}}(-p-cx,-q-ay,1),$$ where
$x,y,p,q$ are evaluated at $(u,v)$. Note that $\psi^*$ extends
analytically to $\R$.

We now consider the analytic \emph{Legendre transform} of
$\psi^*(u,v)$ (see \cite[p. 89]{LSZ}):
\begin{equation}\label{legmap}
\cL (u,v)= \left(-\frac{N_1^*}{N_3^*},-\frac{N_2^*}{N_3^*},-x
\frac{N_1^*}{N_3^*} -y \frac{N_2^*}{N_3^*} -
z^*\right):\Gamma_R\flecha \R^3,
\end{equation}
where $N^*=(N_1^*,N_2^*,N_3^*)$. Since $\psi^*(\Gamma_R)$ is a
locally strictly convex graph, it is well known that so is
$\cL(\Gamma_R)$.

The upwards-pointing unit normal of $\cL$ is
 \begin{equation}\label{norleg}
\mathcal{N}_{\cL}= \frac{(-x,-y,1)}{\sqrt{1+x^2+y^2}}
:\Gamma_R\flecha \S_+^2,
 \end{equation}
where $x,y$ are evaluated at $(u,v)$. Note that $\cL$ is locally
strictly convex and extends analytically to $\R$ with $\cL(u,0)$
lying on the horizontal plane $z=0$, that $\mathcal{N}_{\cL}
(u,0)=(0,0,1)$ and that $z_{xx}^*>0$. This shows that, considering a
smaller $R>0$ if necessary, $\cL(\Gamma_{R})$ lies on the upper
half-space of $\R^3$.

In this way, the intersection of $\cL(\Gamma_R)$ with any horizontal
plane $z=\varepsilon$ for $\varepsilon >0$ small enough is a
regular, strictly convex planar Jordan curve. Also,
$\cL(u,0)=(\gamma(u),0)$ is the limit of those horizonal convex
intersection curves. As $\gamma(u)$ is analytic and non-constant (by
Claim 1), we deduce then that $\gamma(\R)$ bounds a compact strictly
convex set of $\R^2$. This proves Claim \ref{gammacon}.

\end{proof}

\begin{claim}
\label{gammareg} $\gamma$ is a regular curve, that is,
$\gamma'(u)\neq (0,0)$ for every $u\in \R$.
\end{claim}

\begin{proof}[Proof of Claim \ref{gammareg}]
Consider for $\theta\in [0,2\pi)$ the vertical half-plane
$\Pi_{\theta}$ of $\R^3$ with boundary the $z$-axis and which
contains the vector $(\cos \theta, \sin \theta, 0)$. Then, the
intersection of $\Pi_{\theta}$ with the graph $z=z(x,y)$, $(x,y)\in
D^*$, is a regular, real analytic curve, which can be seen as $\psi
(\delta_{\theta})$ where $\delta_{\theta}$ is the real analytic
curve in $\Gamma_R$ given by $$\delta_{\theta}= \{(u,v)\in \Gamma_R
: (x(u,v),y(u,v)) =\landa(u,v) (\cos\theta, \sin \theta) \text{ for
some } \landa (u,v)>0\}.$$

Consider now the map $\cL(u,v):\Gamma_R\flecha \R^3$ defined in
\eqref{legmap}, and let us denote
$\beta_{\theta}:=\cL(\delta_{\theta})$. Note that $\beta_{\t}$ is a
regular, real analytic curve in $\cL(\Gamma_R)$; moreover, it
follows from \eqref{norleg} that $\cL(u,v)\in \beta_{\t}$ if and
only if $-\cN_{\cL} (u,v)\in \S^2\cap \Pi_{\t}$.

Let now $\gamma_{\ep}$ denote the intersection of $\cL(\Gamma_R)$
with the horizontal plane $z=\ep$ in $\R^3$. Then there is some
$\ep_0>0$ small enough such that, for every $\ep \in (0,\ep_0]$,
$\gamma_{\ep}$ is a regular strictly convex Jordan curve, since
$\cL(\Gamma_R)$ is a graph of positive curvature that lies on the
upper half-space of $\R^3$ (see the proof of Claim \ref{gammacon}).
In particular, there exist exactly two points in $\gamma_{\ep}$
where the tangent lines to $\gamma_{\ep}$ are parallel to $(-\sin
\t, \cos \t, 0)$. At these points $\cN_{\cL}$ is orthogonal to
$(-\sin \t, \cos \t, 0)$, which means that either $\cN_{\cL} \in
\S^2\cap \Pi_{\t}$ or $-\cN_{\cL} \in \S^2\cap \Pi_{\t}$ at each of
these points, and that there are no more points in $\gamma_{\ep}$
with this property.

In other words, $\gamma_{\ep}$ intersects $\beta_{\t}\cup
\beta_{\t+\pi}$ at exactly two points, so by continuity we deduce
that for every $\theta\in [0,2\pi)$ and every $\ep\in (0,\ep_0]$
there is some $p_{\ep,\t} \in \R^3$ such that $$\gamma_{\ep}\cap
\beta_{\t} = \{p_{\ep,\t}\}.$$

As explained in the proof of Claim \ref{gammacon}, $\gamma$ is the
limit of the curves $\gamma_{\ep}$ as $\ep \to 0$, and since
$\gamma$ is strictly convex there are exactly two points in $\gamma$
at which the tangent lines to $\gamma$ are parallel to $(-\sin \t,
\cos \t, 0)$. As for every $\ep \in (0,\ep_0]$ the points of
$\gamma_{\ep}$ with that property are exactly
$\{p_{\ep,\t},p_{\ep,\t+\pi}\}$, which lie respectively in
$\beta_{\t}$ and $\beta_{\t+\pi}$, we can deduce that the regular
real analytic curve $\beta_{\t}$ extends continuously to the
boundary of $\cL(\Gamma_R)$. Now, since $\cL(u,v)$ is an immersion
not only in $\Gamma_R$ but also in $\Gamma_R\cup \R$, and since by
definition $\beta_{\t}=\cL(\delta_{\t})$, we deduce that
$\delta_{\t}$ extends continuously to $\Gamma_R\cup \R$ for every
$\t \in [0,2\pi)$.

That is, for every $\theta\in [0,2\pi)$ there is a unique $u_{\t}\in
[0,2\pi)$ such that $\delta_{\t}\cup \{(u_{\t},0)\}$ is a continuous
curve in $\Gamma_R\cup \R$.

We can now finish the proof that $\gamma(u)$ is regular. Assume that
$\gamma'(\overline{u_0}) =(0,0)$ for some $\overline{u_0}\in
[0,2\pi)$. Choose $\t\in [0,2\pi)$ such that
$\overline{u_0}\not\in\{u_{\t},u_{\t+\pi},u_{\t\pm \pi/2}\}$, and
define
$$x_{\theta} (u,v):= \cos \theta x(u,v) +\sin \theta y(u,v),
\hspace{1cm} y_{\theta} (u,v):= -\sin \theta x(u,v) + \cos \theta
y(u,v).$$ Clearly, $x_{\t}(u,0)=y_{\t}(u,0)=0$ for every $u\in \R$,
and from \eqref{d1} we get that $Dx_{\theta} (\overline{u_0},0)=D
y_{\theta} (\overline{u_0},0)=(0,0)$.

Besides, a computation using \eqref{laplacianos} shows that
$x_{\theta} (u,v)$ and $y_{\theta} (u,v)$ satisfy the system of
elliptic PDEs
\begin{equation}\label{laplacianosteta}
   \begin{array}{lcl}
  \Delta x_{\theta}&=& H_1 ((x_{\theta})_u^2+(x_{\theta})_v^2)+H_2 ((x_{\theta})_u (y_{\theta})_u
  +(x_{\theta})_v(y_{\theta})_v) \\&& \\&&+H_3 ((y_{\theta})_u^2+(y_{\theta})_v^2)+H_4 ((x_{\theta})_u (y_{\theta})_v- (x_{\theta})_v (y_{\theta})_u )\\  &&\\
  \Delta y_{\theta}&=& \tilde{H}_1 ((x_{\theta})_u^2+(x_{\theta})_v^2)+\tilde{H}_2 ((x_{\theta})_u (y_{\theta})_u
  +(x_{\theta})_v(y_{\theta})_v) \\&& \\&&+\tilde{H}_3 ((y_{\theta})_u^2+(y_{\theta})_v^2)+\tilde{H}_4 ((x_{\theta})_u (y_{\theta})_v- (x_{\theta})_v (y_{\theta})_u ).
 \end{array}\end{equation}

Here, the coefficients $H_i:=H_i(u,v),
\tilde{H}_i:=\tilde{H}_i(u,v)\in C^{\omega}(\Gamma_R\cup \R)$ are
given in terms of the functions in \eqref{coefhi} by

$$\begin{array}{lcl}
H_1&=& (h_1\circ {\bf z})(\cos \theta)^3+(\tilde{h}_1\circ {\bf z}+h_2\circ {\bf z})(\cos \theta)^2\sin \theta\\&&+(h_3\circ {\bf z}+\tilde{h}_2\circ {\bf z})\cos \theta(\sin \theta)^2 +(\tilde{h}_3\circ {\bf z})(\sin \theta)^3,\\
H_2&=&  (h_2\circ {\bf z})(\cos \theta)^3+(-2h_1\circ {\bf z}+\tilde{h}_2\circ {\bf z}+2h_3\circ {\bf z})(\cos \theta)^2\sin \theta\\
&&-(2\tilde{h}_1\circ {\bf z}+h_2\circ {\bf z}-2\tilde{h}_3\circ {\bf z})\cos \theta(\sin \theta)^2 -(\tilde{h}_2\circ {\bf z})(\sin \theta)^3,\\
H_3&=& (h_3\circ {\bf z})(\cos \theta)^3+(\tilde{h}_3\circ {\bf z}-h_2\circ {\bf z})(\cos \theta)^2\sin \theta\\&&+(h_1\circ {\bf z}-\tilde{h}_2\circ {\bf z})\cos \theta(\sin \theta)^2 +(\tilde{h}_1\circ {\bf z})(\sin \theta)^3,\\
H_4&=& (h_4\circ {\bf z})(\cos \theta)^3+(\tilde{h}_4\circ {\bf z})(\cos \theta)^2\sin \theta+(h_4\circ {\bf z})\cos \theta(\sin \theta)^2 +(\tilde{h}_4\circ {\bf z})(\sin \theta)^3,
 \end{array}$$ and by
 
 $$\begin{array}{lcl}
\tilde{H}_1&=& (\tilde{h}_1\circ {\bf z})(\cos \theta)^3-(h_1\circ {\bf z}-\tilde{h}_2\circ {\bf z})(\cos \theta)^2\sin \theta\\&&+(\tilde{h}_3\circ {\bf z}-h_2\circ {\bf z})\cos \theta(\sin \theta)^2 -(h_3\circ {\bf z})(\sin \theta)^3,\\
\tilde{H}_2&=&  (\tilde{h}_2\circ {\bf z})(\cos \theta)^3-(2\tilde{h}_1\circ {\bf z}+h_2\circ {\bf z}-2\tilde{h}_3\circ {\bf z})(\cos \theta)^2\sin \theta\\
&&+(2h_1\circ {\bf z}-\tilde{h}_2\circ {\bf z}-2h_3\circ {\bf z})\cos \theta(\sin \theta)^2 +(h_2\circ {\bf z})(\sin \theta)^3,\\
\tilde{H}_3&=& (\tilde{h}_3\circ {\bf z})(\cos \theta)^3-(h_3\circ {\bf z}+\tilde{h}_2\circ {\bf z})(\cos \theta)^2\sin \theta\\&&+(\tilde{h}_1\circ {\bf z}+h_2\circ {\bf z})\cos \theta(\sin \theta)^2 -(h_1\circ {\bf z})(\sin \theta)^3,\\
\tilde{H}_4&=& (\tilde{h}_4\circ {\bf z})(\cos \theta)^3-(h_4\circ
{\bf z})(\cos \theta)^2\sin \theta+(\tilde{h}_4\circ {\bf z})\cos
\theta(\sin \theta)^2 -(h_4\circ {\bf z})(\sin \theta)^3.
\end{array} $$ Once here, we can repeat the argument that we did at the beginning of the proof with $x(u,v)$
and $y(u,v)$ using \eqref{laplacianos}, but this time applied to
$x_{\theta}(u,v)$, $y_{\theta} (u,v)$, and using
\eqref{laplacianosteta}. In this way, we conclude that at
$(\overline{u_0},0)$ the real axis is crossed transversely by
another nodal curve of either $x_{\theta}(u,v)$ or
$y_{\theta}(u,v)$. But now we may observe that, by definition of
$\delta_{\theta}$, the nodal set of $x_{\t}(u,v)$ in $\Gamma_R$ is
given by $\delta_{\t}\cup \delta_{\t+\pi}$. As we proved above, this
set can be continuously extended to $\Gamma_R\cup \R$, and the
intersection of $\R$ with this extended set is exactly
$\{(u_{\theta},0),(u_{\t+\pi},0)\}$. In the same way, the nodal set
of $y_{\theta}(u,v)$ in $\Gamma_R$ is $\delta_{\t+\pi/2}\cup
\delta_{\t-\pi/2}$, and its continuous extension intersects the
$v=0$ axis at $\{(u_{\theta+\pi/2},0),(u_{\t-\pi/2},0)\}$.

Since we had chosen $\t\in [0,2\pi)$ such that
$\overline{u_0}\not\in\{u_{\t},u_{\t+\pi},u_{\t\pm \pi/2}\}$,  we
conclude that no nodal curve of $x_{\theta}$ or $y_{\theta}$ can
cross the $u$-axis at $(\overline{u_0},0)$. This contradiction
finishes the proof of Claim \ref{gammareg} and shows that
$\gamma(u)$ is regular (and hence, real analytic).
\end{proof}
\end{proof}

\begin{observacion}
The previous proof shows that the limit gradient $\gamma$ at an
isolated singularity is a regular Jordan curve that is convex. In
particular, the curvature of $\gamma$ does not change sign.
Moreover, for every $p\in \gamma$ the intersection of $\gamma$ with
the tangent line of $\gamma$ at $p$ is exactly $\{p\}$. 

However, the curvature of $\gamma$ can vanish at isolated points. An
example of this phenomenon can be found in \cite[Remark 3]{GJM2}.
\end{observacion}

\section{Classification of isolated singularities of prescribed curvature in space forms}\label{sec:2}

In this section we use the results from Sections \ref{newsec3} and
\ref{secgrad} to classify the embedded isolated singularities of
prescribed, analytic, positive extrinsic curvature in
$3$-dimensional space forms. In Subsection \ref{subsec1} we state
the classification theorem. Subsections \ref{subsec2} and
\ref{subsec3} provide some needed auxiliary results on graphs in
warped products of constant curvature. In Subsection \ref{subsec4}
we prove the classification theorem.

\subsection{Definitions and statement of the classification
theorem}\label{subsec1}

Let $\M^3 =\M^3(c)$ be a $3$-dimensional Riemannian space form of
constant curvature $c\in \{-1,0,1\}$, i.e. $\M^3 = \R^3, \S^3$ or
$\H^3$ depending on whether $c=0,1$ or $-1$, respectively. We will
consider coordinates $(x,y,z)$ on $\M^3(c)$ by making the
identification $\M^3(c) \equiv (\Omega_c,\esiz,\esde)$ where
$\Omega_c=\R^3$ if $c=0,1$, $\Omega_c = \{(x_1,x_2,x_3):
x_1^2+x_2^2+x_3^2<4\}$ if $c=-1$, and $$\esiz, \esde =
\frac{1}{(1+\frac{c}{4}(x_1^2+x_2^2+x_3^2))^2} \,
(dx_1^2+dx_2^2+dx_3^2).$$ For $c=-1$ this is the usual Poincaré ball
model of $\H^3$. For $c=1$, the coordinates $(x_1,x_2,x_3)$ are
given after stereographic projection, and so are defined on
$\S^3\setminus \{\text{north}\}.$

\begin{definicion}\label{defieis}
Given a smooth positive function $\cK:\M^3\flecha (0,\8)$, by an
\emph{embedded isolated singularity of prescribed curvature $\cK$}
in $\M^3$ we mean an immersion $\psi:D\setminus \{q_0\}\flecha \M^3$
from a punctured disk into $\M^3$ such that:
\begin{enumerate}
  \item
The extrinsic curvature of $\psi$ at each point $a\in
\psi(D\setminus \{q_0\})$ is given by $\cK(a)$.
 \item
$\psi$ extends continuously but not $C^1$-smoothly to $D$.
 \item
$\psi$ is an embedding around $q_0$.
\end{enumerate}
\end{definicion}

Let us define for such a surface $\psi:D\setminus \{q_0\}\flecha
\M^3$ the following two notions:

\begin{enumerate}
\item
The \emph{canonical orientation} of $\psi$, that is, the orientation
associated to $\psi$ for which the second fundamental form of $\psi$
is positive definite at every point.
 \item
The \emph{limit unit normal} of $\psi$ at the singularity
$p_0=\psi(q_0)$, i.e. the set $\sigma\subset \{v\in T_{p_0} \M^3 :
|v|=1\}$ given as follows: $v\in T_{p_0} \M^3$ lies in $\sigma$ if
and only if there exists a sequence $\{q_n\}_n$ in $D\setminus
\{q_0\}$ converging to $q_0$ such that the unit normal vectors
$N(q_n)$ of $\psi$ at $q_n$ converge to $v$ in $T\M^3$.
\end{enumerate}

\begin{observacion}
From now on we will assume without loss of generality that the
singularity $p_0$ is placed at the origin $(0,0,0)$ in the model
$(\Omega_c,\esiz,\esde)$ for $\M^3$ explained above. In this way, as
$\esiz,\esde = dx_1^2+dx_2^2+dx_3^2$ at the origin, we see that
$\{v\in T_{(0,0,0)} \M^3 : |v|=1\}$ is canonically identified with
the sphere $\S^2 =\{(x_1,x_2,x_3): x_1^2+x_2^2+x_3^2=1\}$. In
particular, the limit unit normal of $\psi$ at the singularity will
be regarded as a subset of $\S^2$.
\end{observacion}

Once here, we can state the main result of this section.

 \begin{teorema}\label{teo: embedded}
Let $\mathcal{K}:\mathcal{O}\subset\M^3\flecha (0,\8)$ be a positive
real analytic function defined on an open  set $\mathcal{O}\subset
\M^3$ containing a given point $p_0\in\M^3$. Let $\cA_1$ denote the
class of all the canonically oriented surfaces $\Sigma$ in $\M^3$
that have $p_0$ as an embedded isolated singularity, and whose
extrinsic curvature at every point $a\in \Sigma\cap \mathcal{O}$ is
given by $\cK (a)$; here, we identify $\Sigma_1,\Sigma_2\in \cA_1$
if they overlap on an open set containing the singularity $p_0$.

Then, the map that sends each surface in $\cA_1$ to its limit unit
normal at the singularity provides a one-to-one correspondence
between $\cA_1$ and the class $\cA_2$ of regular, analytic, strictly
convex Jordan curves in $\S^2$.
 \end{teorema}

\begin{observacion}
By a \emph{strictly convex} Jordan curve we mean a regular Jordan curve with
the property that its geodesic curvature is non-zero at every point.
\end{observacion}

\begin{observacion}
Theorem \ref{teo: embedded} generalizes \cite[Corollary 13]{GHM}
(which covers the case $\cK= {\rm const.}$ in $\R^3$), \cite[Theorem
4]{GJM2} (for arbitrarily prescribed analytic curvature in $\R^3$)
and \cite[Theorem 15]{GaMi} (for flat surfaces in $\H^3$).

\end{observacion}

\subsection{The prescribed curvature equation in warped
products}\label{subsec2}

Let $\cW\times_f \R$ be   a three-dimensional warped product space,
where $\cW\subset \R^2$ is a neighborhood of a point $p_0\in\cW$
endowed with a conformal metric $g=\lambda(dx^2+dy^2)$ for some
function $\lambda>0$. Then, a computation shows that the extrinsic
curvature $K_{\text{ext}}$ of an immersed graph $z=z(x,y)$ in
$\cW\times_f \R$ is given by the Monge-Ampère equation
\beq\label{general}
   Ar+2B s+C t +rt-s^2=E,\eeq
 for the coefficients

\beq\label{curvwarp}\begin{array}{ll}
A=\displaystyle{\frac{p\lambda_x}{2\lambda}-\frac{q\lambda_y}{2\lambda}-\frac{q^2f'}{f}-\frac{f'}{2}\lambda},&
B=\displaystyle{\frac{p\lambda_y}{2\lambda}+\frac{q\lambda_x}{2\lambda}+\frac{pqf'}{f}},\\
C=\displaystyle{-\frac{p\lambda_x}{2\lambda}+\frac{q\lambda_y}{2\lambda}-\frac{p^2f'}{f}-\frac{f'}{2}\lambda
},&\displaystyle{E=K_{\text{ext}}(f\lambda+p^2+q^2)^2-AC+B^2},
\end{array} \eeq  where $p:=z_x,\ q:=z_y,\ r:=z_{xx},\ s:=z_{xy},\ t:=z_{yy}$.

If we substitute the extrinsic curvature $K_{\text{ext}}$ in
\eqref{curvwarp}  by a smooth function $\cK(x,y,z,z_x,z_y)>0$, then
\eqref{general} becomes a general equation of Monge-Ampère type
\eqref{eq1}, which is elliptic since \beq\label{ellip}
\cD:=AC-B^2+E=K_{\text{ext}}(f\lambda+p^2+q^2)^2>0.\eeq

Therefore, the problem of prescribing the extrinsic curvature of
graphs in warped products depends on solving a general elliptic
equation of Monge-Ampère type.

We focus next on the three-dimensional space forms $\M^3=\M^3(c)$,
as explained in Subsection \ref{subsec1}. It is well known that
these spaces admit several expressions as warped product manifolds,
see e.g. \cite{GJM} for some of them. Here, we will use the
following ones:

\begin{enumerate}
\item[(i)] The hyperbolic $3$-space admits \emph{cylindrical
coordinates} given by the warped product model
 \begin{equation}\label{cilimodel}
 \left(\D_2\times \R, \frac{\cosh^2(z)}{(1-\frac{1}{4}(x^2+y^2))^2} (dx^2+dy^2) +
 dz^2 \right),
 \end{equation}
where $\D_2=\{(x,y): x^2+y^2<4\}$. In this model, the slices
$z={\rm constant}$ are totally geodesic \emph{parallel} hyperbolic
planes, while the vertical lines are geodesics orthogonal to these
totally geodesic planes. Clearly, we can build this model with
respect to any given totally geodesic hyperbolic plane of $\H^2$.

In this model, the graphs $z=z(x,y)$ correspond to geodesic graphs
over totally geodesic hyperbolic planes in the usual sense.
\item[(ii)] Analogously to model $(i)$, we can construct for the
$3$-sphere $\S^3$ a rotationally invariant warped product model, as
$$\S^3\setminus\{\text{north},\text{south}\}\equiv
(\S^2\times (-\pi/2,\pi/2),\cos^2(z) g_{\S^2}+dz^2),$$ where
$g_{\S^2}$ is the standard metric of $\S^2$. Again, this model can
be built with respect to any totally geodesic sphere in $\S^3$. Note
that after stereographic projection from $\S^2$ into $\R^2$, this
model can be written in coordinates as
 \begin{equation}\label{corestres}
\left(\R^2\times (-\pi/2,\pi/2), \frac{\cos^2
(z)}{(1+\frac{1}{4}(x^2+y^2))^2} \, (dx^2+dy^2) + dz^2 \right),
\end{equation}
on the complement of one half of a great circle in
$\S^3\setminus\{\text{north},\text{south}\}$. In these coordinates,
as happened in $(i)$, the graphs $z=z(x,y)$ correspond to geodesic
graphs over totally geodesic spheres of $\S^3$ in the usual sense.
\item[(iii)] In Euclidean space $\R^3$, any choice of orthogonal coordinates $(x,y,z)$
trivially provide warped product coordinates, as
 \begin{equation}\label{corr3}
\left(\R^2\times \R, dx^2+dy^2+dz^2\right).
\end{equation}
\end{enumerate}

Therefore, graphs of positive extrinsic curvature in a  space form
endowed with one of these warped metrics will satisfy
\eqref{general}-\eqref{curvwarp}.

Let us observe that a surface $\Sigma$ in $\M^3$ is a geodesic graph
over some totally geodesic surface $M^2\subset \M^3$ (i.e. $M^2$ is
a plane in $\R^3$, a totally geodesic $\H^2$ in $\H^3$ or a totally
geodesic $\S^2$ in $\S^3$) if and only if $\Sigma$ can be written as
a graph $z=z(x,y)$ for some coordinates $(x,y,z)$ as in
\eqref{cilimodel}, \eqref{corestres} or \eqref{corr3}, with the
surface $M^2$ corresponding in these coordinates to the slice $z=0$.

\subsection{Embedded isolated singularities and
graphs}\label{subsec3}

Let $\Sigma$ be an embedded isolated singularity of positive
extrinsic curvature in $\M^3=\M^3 (c)$, and let $p_0\in \M^3$ denote
the singular point of $\Sigma$.

\begin{lema}
$\Sigma$ is, around $p_0$, a geodesic graph over some totally
geodesic surface $M^2\subset \M^3$.
\end{lema}
\begin{proof}
If $c=0$ (i.e. $\M^3=\R^3$), the result was proved in \cite[Theorem
13]{GaMi}. When $c\neq 0$ (i.e. $\M^3= \S^3$ or $\H^3$), the result
can be reduced to the $c=0$ case. Indeed, recall first of all the
classical result that there exist totally geodesic embeddings from
$\H^3$ and the hemisphere $\S_+^3$ into $\R^3$. These embeddings
preserve geodesics and convexity. In particular, they preserve the
properties of being an embedded isolated singularity of positive
extrinsic curvature, and of being a geodesic graph over some totally
geodesic surface. Therefore, the result in $\H^3$ and $\S^3$ follows
from the result in $\R^3$.
\end{proof}

The previous lemma shows that $\Sigma$ can be viewed around $p_0$ as
a graph $z=z(x,y)$ with respect to a system of coordinates $(x,y,z)$
of $\M^3$ as in \eqref{cilimodel}, \eqref{corestres} or
\eqref{corr3} (depending on whether $c=-1$, $c=1$ or $c=0$,
respectively). We may also assume that the singularity $p_0$
corresponds to $(0,0,0)$ in these coordinates.

We next establish how to relate explicitly the limit unit normal
$\sigma\subset \S^2$ of $\Sigma$ at $p_0$ with the limit gradient
$\gamma\subset \R^2$ of $\Sigma$ at the origin, when $\Sigma$ is
viewed as a graph $z=z(x,y)$ as explained above. Specifically, with
the terminology above, we prove:

\begin{lema}\label{lemonzero}
$\sigma\subset \S^2$ is a regular analytic Jordan curve of
non-vanishing geodesic curvature if and only if $\gamma\subset \R^2$
is so.
\end{lema}
\begin{proof}
Consider the model $(\Omega_c,\esiz,\esde)\equiv \M^3(c)$ for $\M^3$
explained in Subsection \ref{subsec1}.

Let $\Psi: \Omega_c\subset \R^3\flecha \R^3$ be the map defining the
change of coordinates from $(x_1,x_2,x_3)$ to coordinates $(x,y,z)$
as in \eqref{cilimodel}, \eqref{corestres} or \eqref{corr3}, for
which $\Sigma$ is a graph $z=z(x,y)$. By writing $\Psi
(x_1,x_2,x_3)=(x,y,z)$ we observe that $\Psi (0,0,0)=(0,0,0)$, and
that there exists a positively oriented orthonormal basis
$\{e_1,e_2,e_3\}$ of $T_{(0,0,0)} \Omega_c \equiv \R^3$ such that
\begin{equation}\label{changeform}
d\Psi_{(0,0,0)} (e_1)=(1,0,0), \hspace{0.5cm} d\Psi_{(0,0,0)}
(e_2)=(0,1,0) \hspace{0.5cm} d\Psi_{(0,0,0)} (e_3)=(0,0,1).
 \end{equation}

Note that in the $(x,y,z)$ coordinates $\Sigma$ is a graph
$z=z(x,y)$ with an isolated singularity at the origin. Also, observe
that the warped metric in \eqref{cilimodel} (resp.
\eqref{corestres}, \eqref{corr3}) can be written as
 \begin{equation}\label{metc}
f(z) \landa(x,y) (dx^2+dy^2) + dz^2
 \end{equation}
for adequate positive functions $f(z)$, $\landa(x,y)$ depending on
the value of $c$. From here and with this notation, a computation
shows that the unit normal $N$ to $\Sigma$ in these coordinates is
given by

\begin{equation}\label{norxyz}
N(x,y)= \frac{(-z_x,-z_y, f \landa)}{\sqrt{f^2 \landa^2 + f \landa
(z_x^2+z_y^2)}}.
\end{equation}

From this formula, and since $f(0)=\landa(0,0)=1$, we clearly see
that a vector $w=(w_1,w_2)\in \R^2$ is contained in the limit
gradient $\gamma\subset \R^2$ of $z(x,y)$ at the singularity if and
only if the vector given in coordinates $(x,y,z)$ by
$$\frac{1}{\sqrt{1+w_1^2+w_2^2}} \left(-w_1,-w_2,1\right)$$ is contained
in the limit unit normal of $\Sigma$ at the singularity. Using now
\eqref{changeform} we deduce that the limit unit normal
$\sigma\subset \S^2$ in the original $(x_1,x_2,x_3)$ coordinates for
$\M^3$ is given by the set of points $a\in \S^2$ of the form
 \begin{equation}\label{rellim}
a= \frac{1}{\sqrt{1+w_1^2+w_2^2}} \left(-w_1 e_1 -w_2 e_2 +
e_3\right),
 \end{equation}
where $w=(w_1,w_2)$ is a point of the limit gradient $\gamma\subset
\R^2$.

Finally, it is well known (and also easy to prove by a direct
computation) that a curve $\gamma(t)=(\alfa(t),\beta(t))$ in $\R^2$
is regular and has positive (resp. negative) geodesic curvature if
and only if the curve $\sigma(t)\subset \S^2$ given for some
positively oriented orthonormal basis $\{e_1,e_2,e_3\}$ of $\R^3$ by
 \begin{equation}\label{cor2cur}
 \sigma (u)=
\frac{1}{\sqrt{1+\alfa(u)^2 + \beta(u)^2}} \left( -\alfa(u) e_1 -
\beta (u) e_2 + e_3\right)
 \end{equation}
 is regular and has positive (resp. negative) geodesic
curvature in $\S^2$. This fact together with the relation
\eqref{rellim} proves Lemma \ref{lemonzero}.

\end{proof}
\subsection{Proof of the classification theorem}\label{subsec4}

In this subsection we will prove Theorem \ref{teo: embedded}, i.e.
we will show that the map which sends each $\Sigma\in \cA_1$ to its
limit unit normal $\sigma\subset \S^2$ is a bijective correspondence
between $\cA_1$ and $\cA_2$.

To start, let $\Sigma\in \cA_1$, i.e. $\Sigma$ is a canonically
oriented embedded isolated singularity in $\M^3$ of prescribed
analytic curvature $\cK>0$. We assume that the singularity $p_0$ is
placed at $(0,0,0)$ in the canonical coordinates $(x_1,x_2,x_3)$ for
$\M^3$ explained in Subsection \ref{subsec1}, and we denote by
$\sigma$ the limit unit normal to $\Sigma$ at the singularity.

The key point of the proof is the following claim:

\begin{claim}\label{gammareg3}
In the conditions above, $\sigma$ is a regular, analytic, strictly
convex Jordan curve in $\S^2$, i.e. $\sigma\in \cA_2$.
\end{claim}
\begin{proof}[Proof of Claim \ref{gammareg3}]

As explained in Subsection \ref{subsec3}, the surface $\Sigma$ can
be seen around $(0,0,0)$ as a graph $z=z(x,y)$ with an isolated
singularity at the origin with respect to the coordinates $(x,y,z)$
in \eqref{metc}. Moreover, as $\Sigma$ has prescribed extrinsic
curvature $\cK$, it follows that $z(x,y)$ satisfies the elliptic
equation of Monge-Ampère type given by
\eqref{general}-\eqref{curvwarp} with $K_{\ext}= \cK(x,y,z(z,y))>0$.

Besides, observe that in equation \eqref{general}-\eqref{curvwarp},
the functions  $A_p$, $A_q+2B_p$, $C_p+2B_q$ and $C_q$  do not
depend on $p$ and $q$. By Lemma \ref{lem: estrella}, this implies
that $z(x,y)$ satisfies the HeB-condition in Definition \ref{HeB}.
Hence, we are in the conditions to use our analytic study of Section
\ref{secgrad}. In particular, we can parametrize $\Sigma$ around
$(0,0,0)$ with respect to the coordinates $(x,y,z)$ as an analytic
map

\begin{equation}\label{mappsi}
\psi(u,v)=(x(u,v),y(u,v),z(u,v))
\end{equation}
defined on a quotient strip $\Gamma_R=\{w\in \C : 0 < {\rm Im} w <
R\}/(2\pi \Z)$, so that $\psi$ extends analytically to $\R$ with
$\psi(u,0)=(0,0,0)$. Here, the coordinates $(u,v)$ are conformal for
the Riemannian metric $\ep ds^2$ given by \eqref{metrica} in terms
of the coefficients $A,B,C$ in \eqref{curvwarp}.

From the expression \eqref{metc} for the metric of $\M^3$ in the
$(x,y,z)$ coordinates we see that the unit normal of $\psi$ is given
in these coordinates by

 \begin{equation}\label{unitnormal}
 N(u,v)=\frac{(-p(u,v),-q(u,v),f\lambda)}{\sqrt{f^2\lambda^2+f\lambda(p(u,v)^2+q(u,v)^2)}},
 \end{equation}
where $f(u,v):= f(z(u,v))$ and $\landa(u,v):=\landa(x(u,v),y(u,v))$.

By Lemma \ref{extanalitica}, $\psi,N$ are $C^{\omega}$ in $
\Gamma_R\cup \R$. Moreover, as $f(u,0)=\landa(u,0)=1$ we see that
$$N(u,0)=\frac{(-\alfa(u),-\beta(u),1)}{\sqrt{1+\alfa(u)^2+\beta(u)^2}},$$ where we are denoting
$\alfa(u)=p(u,0)$, $\beta(u)=q(u,0)$. In particular, as the limit
gradient at the singularity is given by
$\gamma(u)=(\alfa(u),\beta(u))$, it follows from Theorem
\ref{gammareg} that $N(u,0)$ is a $2\pi$-periodic regular analytic
curve in $\{(x,y,z): x^2+y^2+z^2=1\}$.

At this point, a straightforward computation using the expression
\eqref{unitnormal} for the unit normal of $\psi(u,v)$ shows that the
metric $\ep ds^2$ is conformally equivalent to the second
fundamental form $II$ of $\psi(u,v)$. Note that $II$ is positive
definite since $\Sigma$ is canonically oriented. Thus, if we write
$w=u+iv$, we can express the first and second fundamental forms of
$\psi(u,v)$ as
\begin{equation}\label{fundamx}
\left\{\def\arraystretch{1.5} \begin{array}{rrc} I=\esiz d\psi,d\psi
\esde& =& Q \, dw^2 + 2\mu |dw|^2 +
\overline{Q} d\overline{w}^2, \\
II=-\esiz d\psi,dN\esde&=& 2\rho |dw|^2,
\end{array}\right.
\end{equation}
where $Q:\Gamma_R\cup \R\flecha \C$ is given by $Q:= \esiz
\psi_w,\psi_w\esde$ (here we are denoting
$\parc_w:=(\parc_u-i\parc_v)/2$), and $\mu,\rho :\Gamma_R\cup \R
\flecha (0,\8)$ are a pair of positive real functions. Note that
$Q,\mu,\rho$ are $C^{\omega}$ in $\Gamma_R\cup \R$ since $\psi,N$
have that property.

By \eqref{ellip}, the  extrinsic curvature $K_{\rm ext}$ of
$\psi(u,v)$ is $$K_{\rm ext}(u,v)=\frac{\cD\circ {\bf z}
(u,v)}{(f\lambda+p(u,v)^2+q(u,v)^2)^2},$$ where ${\bf z}(u,v)$ is
given by \eqref{zneg}. In particular, $K_{\rm ext}(u,v)$ extends
analytically as a positive function to $\Gamma_R\cup \R$.

Denote $K:= K_{\rm ext} (u,v)$, and let $\times$ be the exterior
product associated to the warped metric $\langle,\rangle$ in
\eqref{metc}. Then, a direct calculus using \eqref{unitnormal} shows
that

\begin{equation}\label{provex}
N\times N_u = -\sqrt{K} \psi_v, \hspace{1cm} N\times N_v = \sqrt{K}
\psi_u,
\end{equation}
hold in $\Gamma_R\cup \R$. Therefore,
\begin{equation}\label{forQ} \def\arraystretch{2}\begin{array}{lll}
Q(u,0)&=&\displaystyle\frac{1}{4}\left(\esiz \psi_u,\psi_u\esde -
\esiz \psi_v,\psi_v \esde -2 i \esiz \psi_u,\psi_v\esde\right)(u,0)
\\ & =& -\displaystyle\frac{1}{4}\esiz\psi_v,\psi_v\esde (u,0)=
\displaystyle\frac{-1}{4K}\esiz N\times N_u,N\times N_u\esde (u,0)
\\ & = & \displaystyle\frac{-1}{4K} \esiz N_u,N_u\esde (u,0).
\end{array}\end{equation}
Recall that $N(u,0)$ is a regular curve. Then, by choosing a smaller
$R>0$ if necessary, we may assume that $Q$ never vanishes on
$\Gamma_R\cup \R$. Hence, since $K={\rm det} (II)/{\rm det (I)}$ on
$\Gamma_R$, we obtain from \eqref{fundamx} that
\begin{equation}\label{roca}
\rho^2=K(\mu^2-|Q|^2) \hspace{1cm} \text{ in } \Gamma_R\cup \R.
\end{equation}
This shows the existence of a real analytic function $\omega$ in
$\Gamma_R\cup \R$ such that $\mu = |Q|\cosh \omega$ and
$\rho=\sqrt{K} |Q| \sinh \omega$. We observe that $\omega>0$ on
$\Gamma_R$ and $\omega(u,0)=0$ for every $u\in \R$. In particular,
\eqref{fundamx} can be rewritten as
\begin{equation}\label{fundamxx}
\left\{\def\arraystretch{1.5} \begin{array}{rrc} I=\esiz d\psi,d\psi
\esde& =& Q \, dw^2 + 2|Q|\cosh \omega |dw|^2 +
\overline{Q} d\overline{w}^2, \\
II=-\esiz d\psi,dN\esde&=& 2\sqrt{K}|Q|\sinh\omega |dw|^2.
\end{array}\right.
\end{equation}
If $c=0$, it was shown in \cite{Bob}, pp. 118-119 (see also
\cite{GJM2}) that the Gauss-Codazzi equations for $\psi$ imply that
the function $\omega$ satisfies

\begin{equation}\label{pdebobx}
\omega_{w\overline{w}} + U_{\overline{w}} -V_{w} + K |Q| \sinh \omega =0,
\end{equation}
where
 \begin{equation}\label{uveq}
U=\frac{-K_{\overline{w}} Q}{4K |Q|} \sinh \omega, \hspace{1cm}
V=\frac{K_{ w} \overline{Q}}{4K |Q|} \sinh \omega.
 \end{equation}
When $c\neq 0$ the Codazzi equation associated to \eqref{fundamxx}
does not vary, while the Gauss equation gives $K_{\ext} +c = K_{G}$
where $K_{G}$ is the Gaussian curvature of $I$. From here and
\eqref{pdebobx}, we easily see that for any $c\in \{-1,0,1\}$, the
function $\omega$ verifies

\begin{equation}\label{pdebob}
\omega_{w\overline{w}} + U_{\overline{w}} -V_{w} + (K+\varepsilon) |Q| \sinh
\omega =0,
\end{equation}
where $U,V$ are given by \eqref{uveq}. Once here, we see that
\eqref{pdebob} is an elliptic PDE of the type
 \begin{equation}\label{laplaom}
\Delta \omega +a_1\, \omega_u \cosh \omega +  a_2 \, \omega_v \cosh
\omega + a_3\sinh \omega =0,
\end{equation}
where $a_i=a_i(u,v)\in C^{\omega}(\Gamma_R\cup \R)$. Also, note that
$\omega=0$ is a solution to \eqref{laplaom}.

Denote $\eta(u):=N(u,0)$, and observe that the exterior product
$\times$ of $\M^3$ at the origin is just the usual vector product of
$\R^3$ in $(x,y,z)$ coordinates, since the metric \eqref{metc} at
the origin is written as $dx^2+dy^2+dz^2$. Then, by \eqref{provex},
\eqref{forQ},

$$\def\arraystretch{1.5}\begin{array}{lll}
\esiz \eta'',\eta\times \eta'\esde & = & \esiz N_{uu}, N\times
N_u\esde (u,0)= \sqrt{K} \esiz N\times \psi_{uv},N\times N_u\esde
(u,0) \\ & =& \sqrt{K} \esiz \psi_{uv},N_u\esde (u,0)= \sqrt{K}
\left( \frac{\parc}{\parc v} (\esiz \psi_u,N_u\esde ) - \esiz
\psi_u,N_{uv}\esde \right) (u,0) \\ &=& \sqrt{K} \frac{\parc}{\parc
v} (\esiz \psi_u,N_u\esde) (u,0) = - 2 K |Q|  \omega_v \cosh \omega
(u,0)= -2 K |Q|  \omega_v (u,0) \\ &=& -\frac{1}{2} \esiz
\eta',\eta'\esde \omega_v (u,0).
\end{array}$$
Therefore, \beq\label{omegapos}\omega_v(u,0) = -\frac{2\esiz
\eta''(u),\eta(u)\times \eta'(u)\esde}{\esiz
\eta'(u),\eta'(u)\esde}.\eeq  Let us recall at this point that the
real axis is a nodal curve of $\omega$. Since $\omega$ is a solution
to the elliptic PDE \eqref{laplaom}, by Theorem $\dag$ in
\cite{HaWi} we deduce that, at the points $(u,0)$ where $\omega_v
(u,0)= 0$ there exists at least one nodal curve of $\omega$ that
crosses the real axis at a definite angle. But this situation is
impossible, since $\omega>0$ in $\Gamma_R$. Therefore we see that
$\omega_v (u,0)> 0$ for every $u$, and so by \eqref{omegapos} we see
that $\eta(u)$ is a regular, analytic Jordan curve in $\S^2$ of
strictly negative geodesic curvature at every point.

But now, observe that $\eta(u)=N(u,0)$ is simply the expression with
respect to the coordinates $(x,y,z)$ of the limit unit normal of
$\Sigma$ at the singularity. Then, as explained in Subsection
\ref{subsec3}, the limit unit normal $\sigma\subset \S^2$ in the
canonical initial coordinates $(x_1,x_2,x_3)$ of $\M^3$ is given by
$$\sigma(u)= \eta_1 (u) e_1 + \eta_2 (u) e_2 + \eta_3 (u) e_3$$ for some positively oriented
orthonormal basis $\{e_1,e_2,e_3\}$ of $\R^3$. Thus, $\sigma(u)$ is
also a regular analytic Jordan curve in $\S^2$ of strictly negative
geodesic curvature at every point.  This proves Claim
\ref{gammareg3}.
\end{proof}

Claim \ref{gammareg3} shows that the mapping sending each $\Sigma\in
\cA_1$ to its limit unit normal $\sigma\subset \S^2$ is a well
defined mapping from $\cA_1$ to $\cA_2$. So, in order to prove
Theorem \ref{teo: embedded} it remains to check that this map is
bijective.

Surjectivity is a consequence of \cite[Corollary 1]{GJM2}, as
follows.

Consider $\sigma\in \cA_2$, parametrized as a regular, analytic
$2\pi$-periodic curve $\sigma(u):\R/2\pi \Z\flecha \S^2$ of negative
geodesic curvature. Let $\{e_1,e_2,e_3\}$ be a positively oriented
orthonormal frame
 of $\R^3$ for which $\sigma(u)$ can be written as
\eqref{cor2cur} for some $\alfa(u),\beta(u):\R/2\pi \Z \flecha \R$.
It follows then from the last part of the proof of Lemma
\ref{lemonzero} that $\gamma(u):= (\alfa(u),\beta(u))$ is a regular,
analytic $2\pi$-periodic curve of negative geodesic curvature in
$\R^2$.

Consider coordinates $(x,y,z)$ on $\M^3$ as in \eqref{metc}
associated to the basis $\{e_1,e_2,e_3\}$ (i.e. so that
\eqref{changeform} holds), and the elliptic equation of Monge-Ampère
type \eqref{general}-\eqref{curvwarp} with $K_{\ext}= \cK(x,y,z)>0$.
This equation has analytic coefficients. Therefore, item (4) in
\cite[Theorem 3]{GJM2} ensures that there is a solution $z(x,y)$ to
this equation with an isolated singularity at the origin, and whose
limit gradient at the singularity is the curve $\gamma$. Moreover,
$z_{xx}+C>0$ holds, which means that the second fundamental form of
the graph $z=z(x,y)$ is positive definite with respect to its usual
orientation. In other words, the graph $z=z(x,y)$ in $\M^3$ is a
canonically oriented embedded isolated singularity $\Sigma$ of
prescribed curvature $\cK$ around $(0,0,0)$. Moreover, it follows
then from the constructive procedure in the proof of Lemma
\ref{lemonzero} that the limit unit normal of $\Sigma$ at the
singularity is precisely the curve $\sigma\subset \S^2$ we started
with. This proves that the considered map $\cA_1 \to \cA_2$ is
surjective.

We only have left to prove injectivity. Let $\Sigma_1,\Sigma_2\in
\cA_1$ with the same limit unit normal $\sigma\subset \S^2$ at the
singularity $(0,0,0)$. By the results in Subsection \ref{subsec3} it
follows that there exist coordinates $(x,y,z)$ in $\M^3$ as in
\eqref{metc} so that $\Sigma_1,\Sigma_2$ are graphs $z=z_1(x,y),
z=z_2(x,y)$ with an isolated singularity at the origin, and with the
same limit gradient $\gamma$ at the singularity. As both $z_1,z_2$
satisfy the elliptic equation of Monge-Ampère type
\eqref{general}-\eqref{curvwarp} with $K_{\ext}= \cK(x,y,z)>0$, it
follows that both of them also satisfy the HeB-condition (by Lemma
\ref{lem: estrella}). Therefore, both $z_1,z_2$ are in the
conditions of our analysis in Section \ref{secgrad}. In particular,
we can consider for both $z_1,z_2$ the associated maps ${\bf z}_1
(u,v)$, ${\bf z}_2 (u,v)$ given by \eqref{zneg}. Note that $${\bf
z}_1 (u,0)= (0,0,0,\alfa_1(u),\beta_1(u)), \hspace{1cm} {\bf z}_2
(u,0)= (0,0,0,\alfa_2(u),\beta_2(u)),$$ where
$(\alfa_1(u),\beta_1(u))$ and $(\alfa_2(u),\beta_2(u))$ are regular
parametrizations of $\gamma$. Also note that we showed in our proof
above of the fact that the map $\cA_1\to \cA_2$ is well defined that
the parametrization of the limit gradient given by $(p(u,0),q(u,0))$
is oriented so that it has negative curvature. Thus, both
$(\alfa_i(u),\beta_i(u))$, $i=1,2$, have this property. So, up to an
orientation preserving reparametrization of one of them, which
simply means a $2\pi$-periodic conformal reparametrization of the
parameters $(u,v)$ in $\Gamma_R$ (by definition, the parameters
$(u,v)$ were defined up to this type of conformal
reparametrization), we get
$${\bf z}_1 (u,0)= {\bf z}_2 (u,0) =(0,0,0,\alfa(u),\beta(u)),$$
where $(\alfa(u),\beta(u))$ is a regular, analytic parametrization
of $\gamma$. It also follows from \eqref{d1} and \eqref{zfacil0}
that $({\bf z}_1)_v (u,0)= ({\bf z}_2)_v (u,0)$. Hence, both ${\bf
z}_1, {\bf z}_2$ are solutions to the analytic elliptic system
\eqref{mixto} with the same analytic initial conditions. By
uniqueness of the solution to the Cauchy problem for \eqref{mixto},
we get ${\bf z}_1 (u,v)= {\bf z}_2 (u,v)$. In particular, by looking
at the first three coordinates of this equality, we get that the
graphs $z=z_1(x,y)$ and $z=z_2(x,y)$ (i.e. suitable subsets of
$\Sigma_1$ and $\Sigma_2$) overlap on a neighborhood of the
singularity. This proves injectivity and finishes the proof of
Theorem \ref{teo: embedded}.

\begin{observacion}
Given a regular, analytic, strictly convex Jordan curve
$\sigma\subset \S^2$ and some analytic positive function $\cK$ on
$\M^3$, the proof of the surjectivity of the map $\cA_1\flecha
\cA_2$ in the proof of Theorem \ref{teo: embedded} actually provides
a construction process of the unique canonically oriented embedded
isolated singularity $\Sigma$ around $(0,0,0)\in \M^3$ with
prescribed extrinsic curvature $\cK$ which has $\sigma$ as its limit
unit normal at the singularity.
\end{observacion}

\end{document}